\documentclass{article}\author{Gianluca Gorni\\ Universit\`a di Udine\\ Dipartimento di Matematica e Informatica\\ via delle Scienze~208, 33100 Udine, Italy\\ \texttt{gianluca.gorni@uniud.it}\\[10pt] Gaetano  Zampieri\\ Universit\`a di Verona\\ Dipartimento di Informatica\\ strada Le Grazie 15, 37134 Verona, Italy\\ \texttt{gaetano.zampieri@univr.it}}

\usepackage{amsmath,amssymb,amsthm}
\allowdisplaybreaks
\newcommand{\R}{\mathbb{R}}
\newcommand{\N}{\mathbb{N}}

\newtheorem{Theorem}{Theorem}
\newtheorem{Lemma}{Lemma}
\theoremstyle{remark}
\newtheorem{Observation}{Observation}
\newtheorem{Example}{Example}
\theoremstyle{definition}
\newtheorem{Definition}{Definition}

\title{Revisiting Noether's Theorem\\ on constants of motion}

\begin{document}

\maketitle

\begin{abstract}
In this paper we revisit Noether's theorem on the constants of motion for Lagrangian mechanical systems in the ODE case, with some new perspectives on both the theoretical and the applied side. We make full use of invariance up to a divergence, or, as we call it here, Bessel-Hagen (BH) invariance. By recognizing that the Bessel-Hagen (BH) function need not be a total time derivative, we can easily deduce nonlocal constants of motion. We prove that we can always trivialize either the time change or the BH-function, so that, in particular, BH-invariance turns out not to be more general than Noether's original invariance. We also propose a version of time change that simplifies some key formulas. Applications include Lane-Emden equation, dissipative systems, homogeneous potentials and superintegrable systems. Most notably, we give two derivations of the Laplace-Runge-Lenz vector for Kepler's problem that require space and time change only, without BH invariance, one with and one without use of the Lagrange equation.
\end{abstract}

%
%Keywords: Noether, invariance, constants of motion, Bessel-Hagen.

%AMS subject classification: 34C14.

%%%%%%%%%%%%%%%%%%%%%%%%%%%%%%
\section{Introduction}

Emmy Noether's theorem on conserved quantities in mechanical systems is widely celebrated as a fundamental tool in modern theoretical physics. Noether's original 1918 work~\cite{NoetherOriginal} was motivated by general relativity, and to this day most applications of the theorem are in field theory. Mathematically oriented treaties that prove Noether's theorem, such as the ones by Olver~\cite{Olver} and Giaquinta and Hildebrandt~\cite{GiaquintaHildebrandt}, are in a partial differential equation setting and use advanced tools. Seminal papers in the field are the ones by Trautman~\cite{Trautman} and Rund~\cite{Rund}. For historical accounts of how Noether's theorem has been used, the reader can refer for example to Sarlet and Cantrijn's review~\cite{SarletCantrijn}, and the books by Giaquinta and Hildebrandt~\cite{GiaquintaHildebrandt},  Olver~\cite{Olver}, Neuenschwander~\cite{Neuenschwander} and Kosmann-Schwarzbach~\cite{Kosmann}.

It is lamentable that, when we come to classical mechanics with finite degrees of freedom, most textbooks give short shrift to Noether's theorem. For example, Arnold~\cite{Arnold} only proves an easier special case, that is not powerful enough to yield the conservation of energy for autonomous systems. Leach~\cite{Leach1,Leach2} deplores that, all too often, weak versions of Noether's theorem are presented as the real thing. There have been some authors that tried to restore Noether's theorem to its rightful power in introductory classical mechanics, as for example L\'evy-Leblond~\cite{Leblond} and Desloge and Karch~\cite{DeslogeKarch}.

The present work proposes yet one more rethinking of the theory, and a systematization and expansion of the applications. We have two modest claims to breaking a little new ground in this well-tilled field of knowledge. The first one is methodological: we assume the very basic notions of a beginner's course in calculus of variations in one independent variable (Hamiltonian action, Lagrange equation), but we do not use the geometric language of manifolds, vector fields and generalized vector fields. The other point is that, in spite of our severely restricted toolbox, we both reobtain classical results in a different way and naturally extend the theory of constants of motion in the nonlocal direction.

To start with, we have isolated in Lemma~\ref{lemmaOfIntegralFreeFormula} of Subsection~\ref{definitionOfSpaceChange} a step that is in common in all the various formulations of Noether's theorem: it is a basic formula about how the Hamiltonian action changes when the first variation is allowed to have variable endpoints. The simplest version of Noether's theorem (Subsection~\ref{NoetherWithSpaceChangeOnlySubsection}) will be an immediate corollary of that, assuming that the action is (infinitesimally) invariant under space change.

Invariance ``up to a divergence'' was Bessel-Hagen's 1921 generalization~\cite{BesselHagen} of the strict invariance of the original paper by Noether. Since the ``divergence'' only makes sense in a PDE setting, we are going to call this concept ``BH-invariance'', or ``invariance up to a BH-function''. For its ease of use and full power, BH-invariance under space change is our favourite setting for Noether's theorem, of which we will give three variants in Section~\ref{gaugeInvarianceSection}. The version in Theorem~\ref{NoetherWithSpaceChangeAndGauge} is the most natural to justify and to prove. However, we came to discover that it was more general than we had initially expected: it allows for a constant of motion that is not a point function of time, position and velocity, since it involves an integral. The nonlocal constant of motion looks like this:
\begin{equation}\label{integralConsantOfMotionWithGauge}
  t\mapsto
  \partial_{\dot q}
  L\bigl(t,q(t),\dot q(t)\bigr)\cdot
  \partial_\varepsilon q_\varepsilon(t)
  \big|_{\varepsilon=0}
  -\int_{t_0}^t
  \frac{\partial}{\partial\varepsilon}
  L\bigl(s,q_\varepsilon(s),\dot q_\varepsilon(s)\bigr)
  \Big|_{\varepsilon=0}ds\,,
\end{equation}
and does not require special assumptions on the variation~$q_\varepsilon$. The formal statement is in Theorem~\ref{theoremOfVeryGeneralConstantOfMotion} of Subsection~\ref{integralConstantsOfMotionSubsection}. Nonlocal Noetherian constants of motion have been derived for example by Govinder and Leach~\cite{GovinderLeach1995}, but starting from transformations that were themselves nonlocal, while here we get them also from $q_\varepsilon$ that are point functions of $t,q,\dot q$.

We found some neat applications of this theorem: dissipative system (Subsection~\ref{integralConstantsOfMotionSubsection} and Section~\ref{dissipativeSystemsSection}) and the conservative systems with homogeneous potential (Section~\ref{homogeneousPotentialsSection}) exhibit a constant of motion that involves energy and Hamiltonian action. These results seem to be new. Conceivable uses of integral constants of motion may include the study of asymptotic properties of the solution (for example, in Section~\ref{LaneEmdenSection} we find a Lyapunov function for some classes of the Lane-Emden equation), and the design of numerical methods for solving differential equations that preserve a known integral constant of motion, just as there are methods that preserve point-function first integrals already.

Theorem~\ref{NoetherTheoremWithTotalDerivative} in Subsection~\ref{totalDerivativeSubsection} is our third version of Noether's theorem for BH-invariance under space change. It imposes restrictions on the BH-function (the ``total derivative condition'') and on the space change, which ensure that the constant of motion will be an old familiar function of $t,q,\dot q$. We will follow the unifying scheme of Theorem~\ref{NoetherTheoremWithTotalDerivative} in many of the later examples. 

In Section~\ref{GaugeSpaceAndTimeSection} we show how time change can be added to Noether's theorem while avoiding the language of vector fields. Perhaps surprisingly, we have found that time change can be done in two different ways: the one in Subsection~\ref{alternativeApproachToTimeSubsection} is more complicated  but it produces formulas for first integrals (formula~\eqref{constantAlongMotionForAlternativeSpaceTimeChangeAndGauge}) that match exactly what is found in textbooks; the one described in Subsection~\ref{invariancesWithSpaceAndTimeChangeAndGaugeSubsection} seems to be new, it is easier to describe and it leads to simpler expressions for first integrals (formula~\eqref{constantAlongMotionForSpaceTimeChangeAndGauge}). Actually, the two methods are perfectly equivalent (Theorem~\ref{equivalenceOfAlternativeInvariance}) and all examples can be worked out in both ways, one or the other being somewhat more elegant depending on the situation at hand. In this paper we have chosen to give priority to the newer approach in the main theory and in the examples, relegating the more standard one to Subsection~\ref{alternativeApproachToTimeSubsection} and to part of Section~\ref{sectionRungeLenz}.

Beside the two interpretation of time change, Section~\ref{GaugeSpaceAndTimeSection} will also deal with the interchangeability of time change and BH-functions. We can derive constants of motion, as did Emmy Noether, with space and time change only (Theorem~\ref{originalNoetherWithSpaceAndTime}), but also with BH-invariance under space and time change (Theorems~\ref{fullNoetherWithSpaceAndTimeAndGauge} and~\ref{alternativeFullNoetherWithSpaceAndTimeAndGauge}). On the face of it, Bessel-Hagen's formulation would appear the most general. Boyer~\cite{Boyer} in~1967 noticed that we can always ``trivialize time change'' with a simple transformation, and obtain the same constants of motion using space change and a modified BH-function. Olver~\cite[Exercise~5.33]{Olver} remarks without details that for each conservation law of the Euler-Lagrange equation there is a corresponding strict (i.e., with time change and null BH-function) variational symmetry which gives rise to it via Noether's theorem. Our contribution here (Theorems~\ref{EquivalenceOfThreeVersionsTheorem} and~\ref{EquivalenceOfThreeVersionsUsingInverseOfThetaTheorem}) is an explicit formula that ``trivializes the BH-function'' if we are given space change, time change and BH-function, and get the results via a modified time change and null BH-function. In other words, Emmy Noether's formulation is not less general than Bessel-Hagen's version because of a direct, constructive transformation, and not just through the non-elementary converse of Noether's theorem. In each single example it will be a matter of taste which one to use: either time change or BH-function, or both together.

In the second half of the paper we will review several examples of applications of our version of Noether's theorem to classical mechanics systems. We do not have a general symmetry theory behind our choices of space changes, time changes and BH-functions. Some will be familiar from the literature and the others have been found by formula inspection, case by case. We hope that readers with a different background will detect geometric patterns underlying these calculations.

In Section~\ref{EnergyMomentumAngularMomentumSection} we do not add anything new to the classical discussion of momentum and angular momentum, which need neither time change nor BH-function. With conservation of energy we highlight how the time change and the BH-function approach relate to each other, and also how Theorem~\ref{NoetherWithSpaceChangeAndGauge} allows for a BH-function that is not a point function of~$t,q,\dot q$.

In Section~\ref{harmonicOscillatorSection} we take the very familiar harmonic oscillator system and try out some space changes~$q_\varepsilon$ on it, one of them nonlocal. This is intended as an exercise and also to address with examples a superficial objection to our theory: there would be too many constants of motion, one for each arbitrary space change~$q_\varepsilon$ we plug into formula~\eqref{integralConsantOfMotionWithGauge}, and from the total derivative condition, which may seem to be too easy to satisfy. The answer is that a random $q_\varepsilon$ will yield a perfectly fine nonlocal constant of motion. There is also no risk of proliferation of first integrals from trivial solutions to the total derivative condition: the resulting first integrals can very well be trivial too.

In Section~\ref{dissipativeSystemsSection} we expand on the already mentioned dissipative systems. When the potential is homogeneous of degree~2 the constant of motion~\eqref{integralConsantOfMotionWithGauge} becomes an ordinary first integral, providing a new approach to a well-known result by Logan~\cite{Logan}.

In Section~\ref{LaneEmdenSection} we derive integral constants of motion for the Lane-Emden equation, using two distinct combinations of space change and BH-function. The constants of motion when $n\ne5$ may be new, and have a useful dynamical interpretation at least when $n$ is odd and~$\ge7$.

In Section~\ref{homogeneousPotentialsSection} we show that for Lagrangians of the form $L=\frac{1}{2}m\lVert\dot q\rVert^2-U(q)$ with homogeneous potential ($U(\lambda q)=\lambda^qU(q)$) there is a constant of motion that involves energy and Hamiltonian action. In the special case~$k=-2$ this constant of motion reduces to a point-function of $t,q,\dot q$ (formula~\eqref{homogeneousPotentialFistIntegral}), which has illuminating consequences for the dynamics (formula~\eqref{homogeneousPotentialRadiusOnTime} in particular). As far as we are aware, these results were well-known to Celestial Mechanics only in the central potential case $U(q)= -k/\Vert q\rVert^2$. They also apply to Calogero's system.

In Sections~\ref{todaLattice} and~\ref{conePotentials} we make some calculations leading to nonlocal constants of motion for the nonperiodic Toda lattice and, more generally, for ``cone potentials''.

In Section~\ref{sectionRungeLenz} we deduce the Laplace-Runge-Lenz vector constant for Kepler's problem in two different ways. The first uses the Lagrange equation, with a space variation~$q_\varepsilon$ which is simpler than usual in the literature. The second approach does not use the Lagrange equation. For both approaches we show how we can trivialize either the BH-function or the time change. We also take the opportunity to show how our formulation of Noether's theorem can compute functions that are constant along single motions, and not necessarily along all solutions to Lagrange equations.

In Section~\ref{sectionisochronous} we use the technique of Theorem~\ref{NoetherTheoremWithTotalDerivative} to extend the class of superintegrable system that was found with a different method in a recent paper~\cite{Zampieri}. Finally, in Section~\ref{PlaneWaveLikeExample} we study the system of a particle in a plane-wavelike field, that we have taken from a paper~\cite{Bobillo} by Bobillo-Ares.

%%%%%%%%%%%%%%%%%%%%%%%%%%
\subsection*{Notations}

Throughout this work, $\R^n$ will be the usual Euclidean space, and we will use the notation $x\cdot y$ for the inner product and $\lVert x\rVert$ for the norm. The symbol $q$ will denote either a vector in $\R^n$ or an $\R^n$-valued function of une variable, and the context should clarify which is the case. Given two points $x=(x_1,x_2)$, $y=(y_1,y_2)$ in the plane~$\R^2$ we set $\det(x,y):=x_1y_2-x_2y_1$. For partial derivatives of expressions with respect to a variable~$x$ we use the fractional notation $\partial/\partial x$, whilst for the partial derivatives of a named function $F$ we write more simply $\partial_xF(x,y)$ or~$\partial^2_{x,y}F(x,y)$. The gradient of a smooth scalar function $q\mapsto f(q)$ of $n$~variables will be denoted by $\partial_qf(q)$, and it will be treated as a vector in~$\R^n$. To clarify our usage without a formal definition, here is Taylor's formula for a scalar function of~$(p,q)\in\R^n\times\R^n$:
\begin{align*}
  f(p+h,q+k)={}&f(p,q)+\partial_pf(p,q)\cdot h+
  \partial_qf(p,q)\cdot k+
  o\bigl(\lVert h\rVert+\lVert k\rVert\bigr)\,.
\end{align*}

%%%%%%%%%%%%%%%%%%%%%%%%%%

\section{Invariance under space change only}
\label{infinitesimalVariationalInvarianceSection}

The variational approach to classical mechanics starts with a smooth Lagrangian function $L(t,q,\dot q)$ defined in an open subset of $\R\times\R^n\times\R^n$ with values in~$\R$, and trajectories $q(t)$ such that $L(t,q(t),\dot q(t))$ is defined, so that we can introduce the Hamiltonian action functional as
\begin{equation}
  A_{a,b}(q):=
  \int_a^b
  L\bigl(t,q(t),\dot q(t)\bigr)\,dt\,.
\end{equation}
We then posit that the mechanical motions will be fixed-endpoint-stationary for the action functional, in the sense of the Calculus of Variations. As well known, the stationarity condition is equivalent to the Lagrange, or Euler-Lagrange, differential equations:
\begin{equation}\label{EulerLagrangeEquation}
  \frac{d}{dt}\partial_{\dot q}
  L\bigl(t,q(t),\dot q(t)\bigr)-
  \partial_qL\bigl(t,q(t),\dot q(t)\bigr)
  =0\qquad\forall t\,.
\end{equation}
Noether's contribution was to show that mechanical conservation laws can be deduced from some kind of invariance property of the action functional. Let us review these invariances, starting from the simplest case.

\subsection{Space change and Hamiltonian action}

\begin{Definition}
\label{definitionOfSpaceChange}
Given a trajectory $q(t)$, we will call \emph{space change} a smooth one-parameter family of trajectories $(\varepsilon,t)\mapsto q_\varepsilon(t)$, with values in~$\R^n$, defined for the same times~$t$ as the motion~$q(t)$, and for $\varepsilon$ close to~0, such that $q_\varepsilon(t)$ reduces to $q(t)$ when $\varepsilon=0$.
\end{Definition}

This notion of space change is general enough to encompass what we call the \emph{time-shift family} $q_\varepsilon(t):=q(t+\varepsilon)$ as a special case. We still call it ``space'' change because the Lagrangian $L(t,q_\varepsilon,\dot q_\varepsilon)$ depends on~$\varepsilon$ only through the space-velocity slots. In Section~\ref{GaugeSpaceAndTimeSection} we will let the time slot depend on~$\varepsilon$ too, by means of what we will call ``time change''. 

Consider the action functional along the family $q_\varepsilon$ as a function of~$\varepsilon$:
\begin{equation}\label{actionOnly}
  \varepsilon\mapsto A_{a,b}(q_\varepsilon)\,.
\end{equation}
One crucial and easy fact is that the derivative of the action with respect to $\varepsilon$ at $\varepsilon=0$ has an integral-free formula if we assume that $q(t)$ is a solution to the Lagrange equations:

\begin{Lemma}[Integral-free formula for the derivative of the action]\label{lemmaOfIntegralFreeFormula}
If $q(t)$ is a solution to the Lagrange equations then
\begin{equation}\label{integralFreeFormula}
  \begin{split}
  \frac{\partial A_{a,b}(q_\varepsilon)}{\partial\varepsilon}
  \Big|_{\varepsilon=0}={}&
  \partial_{\dot q}
  L\bigl(t,q(t),\dot q(t)\bigr)\cdot
  \partial_\varepsilon q_\varepsilon(t)
  \Big|_{\varepsilon=0,t=b}-\\
  {}&-\partial_{\dot q}
  L\bigl(t,q(t),\dot q(t)\bigr)\cdot
  \partial_\varepsilon q_\varepsilon(t)
  \Big|_{\varepsilon=0,t=a}\,.
  \end{split}
\end{equation}
\end{Lemma}

\begin{proof}
Using the Lagrange equation
\begin{align}\label{derivativeOfLagrangianWRTepsilon}
  \frac{\partial}{\partial\varepsilon}
  L\bigl(t,&q_\varepsilon(t),
  \dot q_\varepsilon(t)\bigr)\Big|_{\varepsilon=0}
  =\\
  ={}&\Bigl(
  \partial_q
  L\bigl(t,q(t),\dot q(t)\bigr)\Bigr)
  \cdot\partial_\varepsilon q_\varepsilon(t)
  \big|_{\varepsilon=0}
  +\partial_{\dot q}L\bigl(t,q(t),\dot q(t)\bigr)\cdot
  \bigl(\partial_\varepsilon\dot q_\varepsilon(t)
  \big|_{\varepsilon=0}\bigr)=\notag\\
  ={}&
  \Bigl(
  \frac{d}{dt}\partial_{\dot q}
  L\bigl(t,q(t),\dot q(t)\bigr)\Bigr)
  \cdot\partial_\varepsilon q_\varepsilon(t)
  \big|_{\varepsilon=0}+
  {}\notag\\&+
  \partial_{\dot q}L\bigl(t,q(t),\dot q(t)\bigr)\cdot
  \frac{d}{dt}\bigl(\partial_\varepsilon q_\varepsilon(t)
  \big|_{\varepsilon=0}\bigr)=\notag\\
  ={}&\frac{d}{dt}\Bigl(
  \partial_{\dot q}
  L\bigl(t,q(t),\dot q(t)\bigr)\cdot
  \partial_\varepsilon q_\varepsilon(t)\big|_{\varepsilon=0}
  \Bigr)\,.\notag
\end{align}
Integrating with respect to~$t$ and taking the derivative $\frac{\partial}{\partial\varepsilon}$ out of the integral sign we obtain equation~\eqref{integralFreeFormula}.
\end{proof}

If $q_\varepsilon(a)$ and $q_\varepsilon(b)$ do not depend on~$\varepsilon$, formula~\eqref{integralFreeFormula} is zero. This fact should not surprise, since $q(t)$ by definition makes the action stationary with respect to variations that keep the endpoints fixed. Lemma~\ref{lemmaOfIntegralFreeFormula} simply generalizes to the case when the endpoints of the variation are not fixed.

\subsection{Simplest version of Noether's theorem}
\label{NoetherWithSpaceChangeOnlySubsection}

Classical mechanics textbooks that have chosen to dwell on Noether's theorem as little as possible usually present the following simple version, which only uses invariance under space change (Arnold's book~\cite{Arnold}, for example).

\begin{Definition}
\label{definitionOfInvarianceUnderSpaceChange}
We will say that there is infinitesimal \emph{invariance under the space change}~$q_\varepsilon$ if
\begin{equation}\label{zeroDerivativeOfAction}
  \frac{\partial A_{a,b}(q_\varepsilon)}{\partial\varepsilon}
  \Big|_{\varepsilon=0}
  = 0\qquad\forall a,b,
\end{equation}
or, equivalently, that the Lagrangian is invariant:
\begin{equation}\label{infinitesimalInvarianceSimpleDerivated}
  \frac{\partial}{\partial\varepsilon}
  L\bigl(t,q_\varepsilon(t),\dot q_\varepsilon(t)\bigr)
  \Big|_{\varepsilon=0}=0
  \qquad\forall t\,.
\end{equation}
\end{Definition}

\begin{Theorem}[Noether's theorem with space change only]
\label{simplestTheoremWithSpaceChangeOnly}
Suppose that the Hamiltonian action $A$ is infinitesimally invariant under the space change~$q_\varepsilon$, and that $q(t)=q_0(t)$ is a solution to Lagrange equation. Then the function
\begin{equation}
  t\mapsto\partial_{\dot q}
  L\bigl(t,q(t),\dot q(t)\bigr)\cdot
  \partial_\varepsilon q_\varepsilon(t)
  \big|_{\varepsilon=0}\,.
\end{equation}
is constant in~$t$.
\end{Theorem}

\begin{proof}
Let us set
\begin{equation}\label{functionF}
  F(t)=\partial_{\dot q}
  L\bigl(t,q(t),\dot q(t)\bigr)\cdot
  \partial_\varepsilon q_\varepsilon(t)
  \big|_{\varepsilon=0}\,.
\end{equation}
From Lemma~\ref{lemmaOfIntegralFreeFormula} we have that
\begin{equation*}
  \frac{\partial A_{a,b}(q_\varepsilon)}{\partial\varepsilon}
  \Big|_{\varepsilon=0}=
  F(b)-F(a)\,.
\end{equation*}
Combining this with assumption~\eqref{zeroDerivativeOfAction}, we obtain that $F(b)-F(a)=0$ for all~$a,b$, which means that $F$ is constant.
\end{proof}

There are situations when not only the function $\varepsilon\mapsto A_{a,b}(q_\varepsilon)$ has zero derivative at~$\varepsilon=0$, but it is actually constant, or, equivalently, $L(t,q_\varepsilon(t), \allowbreak \dot q_\varepsilon(t))$ does not depend on~$\varepsilon$. These will be called \emph{finite invariances under space change}. The standard examples are the Lagrangians that are invariant under either translations or rotations (Examples~\ref{translationInvarianceExample} and~\ref{rotationalInvarianceExample} in Section~\ref{EnergyMomentumAngularMomentumSection}), whereby any smooth trajectory $q(t)$ can be embedded in a translated or rotated family~$q_\varepsilon(t)$ with the property of finite invariance. When $q(t)$ also solves the Lagrange equations, the momentum or angular momentum will be conserved respectively. Although very simple and neat, this concept of finite invariance does not seem powerful enough to cover conservation of energy, for instance.

%%%%%%%%%%%%%%%%%%%%%%%%
\section{BH-invariance under space change}
\label{gaugeInvarianceSection}

\begin{Definition}\label{definitionOfGauge}
We will call \emph{BH-function} a smooth real function $G(\varepsilon,t)$, defined for the same times $t$ as the motion~$q(t)$, and for $\varepsilon$ close to~0. The trivial BH-function will simply be the constant~0.
\end{Definition}

Some authors use the term ``gauge function'' for what we call BH-function. We can modify the function~\eqref{actionOnly} by introducing a BH-function $G(\varepsilon,t)$, and then ask that the new function
\begin{equation*}
  \varepsilon\mapsto A_{a,b}(q_\varepsilon)+G(\varepsilon,b)-
  G(\varepsilon,a)
\end{equation*}
has zero derivative at~$\varepsilon=0$, for all~$a,b$. This idea leads to a version of Noether's theorem that is at the same time widely applicable and easy to state.

%%%%%%%%%%%%
\subsection{Noether's theorem in terms of BH-function and space change}
\label{NoetherWithGaugeAndSpaceChangeSubsection}

\begin{Definition}
We will say that there is infinitesimal \emph{BH-invariance under space change}~$q_\varepsilon$, with $G$ as BH-function term if
\begin{equation}\label{actionPlusGauge}
  \frac{\partial}{\partial\varepsilon}\Bigl(
  A_{a,b}(q_\varepsilon)+G(\varepsilon,b)-
  G(\varepsilon,a)\Bigr)\Big|_{\varepsilon=0}=0
  \qquad\forall a,b\,,
\end{equation}
or, in term of the Lagrangian, if
\begin{equation}\label{gaugeAndSpace}
  \frac{\partial}{\partial\varepsilon}
  \Bigl(L\bigl(t,q_\varepsilon(t),\dot q_\varepsilon(t)\bigr)
  +\partial_tG(\varepsilon,t)\Bigr)
  \Big|_{\varepsilon=0}=0\,.
\end{equation}
\end{Definition}

Combined with~\eqref{integralFreeFormula}, this invariance causes the modified function
\begin{equation*}
  t\mapsto F(t)+ \partial_\varepsilon 
  G(\varepsilon,t)\big|_{\varepsilon=0}
\end{equation*}
to be a constant of motion, when $q(t)$ solves the Lagrange equations. Of course we may speak of \emph{finite} BH-invariance in the rare instances when the function~\eqref{actionPlusGauge} is constant (Example~\ref{timeShiftInvarianceExample} in Section~\ref{EnergyMomentumAngularMomentumSection}).

Here is the version of Noether's theorem that uses BH-invariance under space change. It reduces to Theorem~\ref{simplestTheoremWithSpaceChangeOnly} when the BH-function is trivial.

\begin{Theorem}[Noether's theorem with space change and BH-function]
\label{NoetherWithSpaceChangeAndGauge}
Suppose that there is infinitesimally BH-invariance under space change, and that $q(t)=q_0(t)$ is a solution to Lagrange equation. Then the function
\begin{equation*}
  t\mapsto\partial_{\dot q}
  L\bigl(t,q(t),\dot q(t)\bigr)\cdot
  \partial_\varepsilon q_\varepsilon(t)
  \big|_{\varepsilon=0}%+{}\\
  +\partial_\varepsilon 
  G(\varepsilon,t)\big|_{\varepsilon=0}\,.
\end{equation*}
is constant in~$t$.
\end{Theorem}

\begin{proof}
As in Theorem~\ref{simplestTheoremWithSpaceChangeOnly} let us set
$F(t)=\partial_{\dot q} L(t,q(t),\dot q(t))\cdot \partial_\varepsilon q_\varepsilon(t)|_{\varepsilon=0}$. Again from Lemma~\ref{lemmaOfIntegralFreeFormula} we have that
\begin{equation*}
  \frac{\partial}{\partial\varepsilon}\Bigl( A_{a,b}(q_\varepsilon)
  +G(\varepsilon,b)-
  G(\varepsilon,a)\Bigr)
  \Big|_{\varepsilon=0}
  =F(b)-F(a)+\partial_\varepsilon 
  G(0,b)-
  \partial_\varepsilon 
  G(0,a)\,.
\end{equation*}
Combining this with assumption~\eqref{actionPlusGauge}, we obtain that $F(b)-F(a)+\partial_\varepsilon G(0,b)-\partial_\varepsilon G(0,a)=0$ for all~$a,b$, which means that $t\mapsto F(t)+\partial_\varepsilon G(0,t)$ is constant.
\end{proof}

The simplest meaningful example of infinitesimal BH-invariance occurs when the Lagrangian function $L$ does not depend on~$t$. We can embed any smooth $q(t)$ into its time-shift family 
\begin{equation}\label{timeShiftForEnergy}
  q_\varepsilon(t):= q(t+ \varepsilon)\,,
\end{equation}
whose derivatives with respect to~$t$ are the same as the derivatives with respect to~$\varepsilon$. Let us try a BH-function~$G$ of the form $G(\varepsilon, t)= \varepsilon g(t)$. The BH-invariance condition~\eqref{gaugeAndSpace} becomes
\begin{equation}\label{totalDerivativeForEnergyConservation}
  g'(t)=-\Bigl(\frac{\partial}{\partial\varepsilon}
  L\bigl(q_\varepsilon(t),\dot q_\varepsilon(t)\bigr)
  \Bigr)
  \Big|_{\varepsilon=0}=
  -\frac{d}{dt}L\bigl(q(t),\dot q(t)\bigr)\,.
\end{equation}
By inspection we see that \eqref{gaugeAndSpace} holds with the choice
\begin{equation}\label{gaugeForEnergy}
  G(\varepsilon, t):=-\varepsilon\cdot L(q(t), \dot q(t))
\end{equation}
We can deduce that when $q(t)$ solves the Lagrange equations, we have \textit{conservation of energy}. More on this in Example~\ref{timeShiftInvarianceExample} in Section~\ref{EnergyMomentumAngularMomentumSection}.

%%%%%%%%%%%%%%%%%%%%%%%
\subsection{Integral constants of motion}
\label{integralConstantsOfMotionSubsection}

Here is a reformulation of Theorem~\ref{NoetherWithSpaceChangeAndGauge} that bypasses the concept of BH-function.

\begin{Theorem}[Noether's theorem with integral constant of motion]\label{theoremOfVeryGeneralConstantOfMotion}
Suppose that $t\mapsto q(t)$ is a solution to the Lagrange equation. Then the following function is constant:
\begin{equation}\label{veryGeneralConstantAlongMotion}
  t\mapsto
  \partial_{\dot q}
  L\bigl(t,q(t),\dot q(t)\bigr)\cdot
  \partial_\varepsilon q_\varepsilon(t)
  \big|_{\varepsilon=0}-
  \int_{t_0}^t
  \frac{\partial}{\partial\varepsilon}
  L\bigl(s,q_\varepsilon(s),\dot q_\varepsilon(s)\bigr)
  \Big|_{\varepsilon=0}ds\,.
\end{equation}
\end{Theorem}

\begin{proof}
Notice that $A_{a,b}(q_\varepsilon)=A_{t_0,b}(q_\varepsilon)-
A_{t_0,a}(q_\varepsilon)$ for all~$a,b,\varepsilon$. Take the derivative with respect to~$\varepsilon$ at $\varepsilon=0$, using Lemma~\ref{lemmaOfIntegralFreeFormula} and the function~$F$ of formula~\eqref{functionF} on the left-hand side:
\begin{equation*}
  F(b)-F(a)=\frac{\partial}{\partial\varepsilon}
  A_{t_0,b}(q_\varepsilon)-
  \frac{\partial}{\partial\varepsilon}
  A_{t_0,a}(q_\varepsilon)\,.
\end{equation*}
A simple separation of the terms in~$a$ and in~$b$ leads to
\begin{equation*}
  F(b)-\frac{\partial}{\partial\varepsilon}
  A_{t_0,b}(q_\varepsilon)=
  F(a)-\frac{\partial}{\partial\varepsilon}
  A_{t_0,a}(q_\varepsilon)\,,
\end{equation*}
which means that the expression~\eqref{veryGeneralConstantAlongMotion} is constant in~$t$.

We may also apply Theorem~\ref{NoetherWithSpaceChangeAndGauge}. The condition of BH-invariance under space change~\eqref{gaugeAndSpace} is very easy to realize: with the choice $G(\varepsilon,t):= -A_{t_0,t} (q_\varepsilon)$ the function~\eqref{actionPlusGauge} is trivially constant. Alternatively, if we set
\begin{equation}\label{veryGeneralGaugeLinearInEpsilon}
  G(\varepsilon,t):=-\varepsilon\int_{t_0}^t
  \frac{\partial}{\partial\varepsilon}
  L\bigl(s,q_\varepsilon(s),\dot q_\varepsilon(s)\bigr)
  \Big|_{\varepsilon=0}ds
\end{equation}
the function~\eqref{actionPlusGauge} has zero derivative at~$\varepsilon=0$. In either case the constant function $t\mapsto F(t)+ \partial_\varepsilon G(\varepsilon,t)|_{\varepsilon=0}$ turns out to be~\eqref{veryGeneralConstantAlongMotion}.
\end{proof}

We do not expect that all, or most, of the constants of motion that result by plugging arbitrary space changes $q_\varepsilon$ into formula~\eqref{veryGeneralConstantAlongMotion} are distinct, nontrivial or useful. However, take for example a Lagrangian of the form
\begin{equation}\label{dissipativeSystemLagrangian}
  \mathcal{L}(t,q,\dot q):=e^{ht} L(q,\dot q)
\end{equation}
and choose the time-shift space change
\begin{equation}\label{dissipativeSystemSpaceChange}
  q_\varepsilon(t):=q(t+\varepsilon)\,.
\end{equation}
We can compute
\begin{equation}\label{pseudoTotalDerivativeForDissipativeSys}
  \Bigl(\frac{\partial}{\partial\varepsilon}
  \mathcal{L}\bigl(t,q_\varepsilon(t),\dot q_\varepsilon(t)\bigr)
  \Bigr)\Big|_{\varepsilon=0}=
  e^{ht} \frac{d}{dt}L\bigl(q(t),\dot q(t))
  \,.
\end{equation}
An integration by part turns the BH-function~\eqref{veryGeneralGaugeLinearInEpsilon} into the following:
\begin{equation}\label{dissipativeSystemGauge}
  G(\varepsilon,t):=
  -\varepsilon e^{ht} L\bigl(q(t),\dot q(t)\bigr)%+\\
  +\varepsilon h\int_{t_0}^te^{hs }
  L\bigl(q(s),\dot q(s)\bigr)ds\,.
\end{equation}
The constant of motion is computed as
\begin{equation}\label{energyPlusAction}
  e^{ht} \partial_{\dot q}
  L\bigl(q(t),\dot q(t)\bigr)\cdot\dot q(t)-
  e^{ht} L\bigl(q(t),\dot q(t)\bigr)%+\\
  +h\int_{t_0}^te^{hs} L\bigl(q(s),\dot q(s)\bigr)ds\,,
\end{equation}
For a mechanical interpretation of this formula see Section~\ref{dissipativeSystemsSection}. For now it will suffice to notice that this conserved quantity~\eqref{energyPlusAction} is nontrivial and that it depends not just on the values of $t,q(t),\dot q(t)$ at the current time~$t$, but also on the values on the whole interval $[t_0,t]$. Hence, it is not a bona fide first integral. Another class of examples with an integral constant of motion are the homogeneous potentials, as we will see in Section~\ref{homogeneousPotentialsSection}.

%%%%%%%%%%%%%%%%%%%%%%%%%%%
\subsection{The total derivative condition}
\label{totalDerivativeSubsection}

There are few and precious mechanical systems for which Noether's Theorem~\ref{theoremOfVeryGeneralConstantOfMotion} does yield a true first integral, in the sense of a point function of $t,q(t),\dot q(t)$ that is constant along the motions. For the study of these systems there is a unifying general scheme that we are going to describe next, and that will be adhered in Sections~\ref{EnergyMomentumAngularMomentumSection} through~\ref{PlaneWaveLikeExample}, devoted to examples.

\begin{Definition}\label{definitionOfTotalDerivativeCondition}
We will say that the \emph{total derivative condition} is satisfied with the smooth scalar function $\psi(t,q,\dot q)$ if \begin{equation}\label{totalDerivativeCondition}
  \Bigl(\frac{\partial}{\partial\varepsilon}
  L\bigl(t,q_\varepsilon(t),\dot q_\varepsilon(t)\bigr)
  \Bigr)\Big|_{\varepsilon=0}=
  \frac{d}{dt}\psi\bigl(t,q(t),\dot q(t)\bigr)
  \quad\forall t\,.
\end{equation}
\end{Definition}

Formula~\eqref{totalDerivativeForEnergyConservation}, that we wrote to derive conservation of energy, is equivalent to a total derivative condition where $\psi$ is simply~$L$. Compare that with~\eqref{pseudoTotalDerivativeForDissipativeSys}, whose right-hand side does not look like a total derivative.

\begin{Theorem}[Noether's theorem with total derivative condition]\label{NoetherTheoremWithTotalDerivative}
Suppose that  the total derivative condition is satisfied with~$\psi$, and there exists a vector-valued function~$\varphi(t,q,\allowbreak\dot q)$ such that
\begin{equation}\label{conditionOnDerivativeWithRespectToEpsilon}
  \partial_\varepsilon q_\varepsilon(t)\big|_{\varepsilon=0}=
  \varphi\bigl(t,q(t),\dot q(t)\bigr)
  \qquad\forall t\,,
\end{equation}
and, as usual, that $t\mapsto q(t)$ is a solution to the Lagrange equation. Then the following function is constant if evaluated along $q(t)$:
\begin{equation}\label{firstIntegralForTotalDerivative}
  (t,q,\dot q)\mapsto
  \partial_{\dot q}L(t,q,\dot q)\cdot\varphi(t,q,\dot q)-
  \psi(t,q,\dot q)\,.
\end{equation}
\end{Theorem}

\begin{proof}
Simply apply Theorem~\ref{NoetherWithSpaceChangeAndGauge} with  $G(\varepsilon,t):=-\varepsilon\cdot\psi(t,q(t),\dot q(t))$.
\end{proof}

We did not give a name to condition~\eqref{conditionOnDerivativeWithRespectToEpsilon} because it is automatically satisfied whenever $q_\varepsilon(t)$ is a point function of $\varepsilon, t,q(t),\dot q(t),\allowbreak q(t\pm \varepsilon)$. More generally, a property of all $q_\varepsilon(t)$ of this form is that when $\varepsilon=0$ all their successive partial derivatives with respect to~$\varepsilon$ and~$t$ are point functions of $t,q(t),\dot q(t),\ddot q(t)$\dots{}  For example for the time-shift family $q_\varepsilon (t)=q(t+\varepsilon)$ of formulas~\eqref{timeShiftForEnergy} and~\eqref{dissipativeSystemSpaceChange} we have $\partial_\varepsilon q_\varepsilon(t)|_{\varepsilon=0}=\dot q(t)$. Actually, we could reach the same final conclusions with space variations of the more special form $q_\varepsilon(t)=q(t)+ \varepsilon \cdot \varphi(t,q(t),\dot q(t))$, but we have we have used a delay term $q(t\pm\varepsilon)$ in some examples when we felt it led to simpler formulas.

Of course, a function $\psi$ that trivially satisfies the total derivative condition~\eqref{totalDerivativeCondition} always exists:
\begin{equation}\label{trivialPsi}
  \psi(t,Q,\dot Q):=\int_{t_0}^t
  \Bigl(\frac{\partial}{\partial\varepsilon}
  L\bigl(s,q_\varepsilon(s),\dot q_\varepsilon(s)\bigr)
  \Bigr)\Big|_{\varepsilon=0}ds
\end{equation}
(with dummy dependence on the variables $Q,\dot Q$), and it leads back exactly to the integral constant of motion~\eqref{veryGeneralConstantAlongMotion}. This trivial $\psi$ of formula~\eqref{trivialPsi} however only works for that particular motion $t\mapsto q(t)$.

The total derivative condition becomes interesting when we manage to find one single function~$\psi$ that satisfies the equation~\eqref{totalDerivativeCondition} \emph{for all smooth paths} $q(t)$ at the same time, whether they solve Lagrange's equations or not. You cannot expect to find such a~$\psi$ for a random choice of~$L$ and~$q_\varepsilon$. In fact, by integrating formula~\eqref{totalDerivativeCondition} we get the following equality
\begin{equation}
  \label{integratedTotalDerivativeCondition}
  \int_{t_0}^{t_1}\Bigl(\frac{\partial}{\partial\varepsilon}
  L\bigl(s,q_\varepsilon(s),\dot q_\varepsilon(s)\bigr)
  \Bigr)\Big|_{\varepsilon=0}ds=
  \psi\bigl(t_1,q(t_1),\dot q(t_1)\bigr)-
  \psi\bigl(t_0,q(t_0),\dot q(t_0)\bigr)
  \quad\forall t\,,
\end{equation}
where the right-hand side only depends on the end values of $q(t),\dot q(t)$, while the left-hand side involves all values for $t\in[t_0,t_1]$ (see Examples~\ref{spaceShiftForOscillator} and~\ref{spaceDilationForOscillator} in Section~\ref{EnergyMomentumAngularMomentumSection}).

If we restrict our attention to the Lagrangian motions only, the following choice for~$\psi$ always works:
\begin{equation}\label{trivialPsi2}
  \psi(t,q,\dot q):=
  \partial_{\dot q}L(t,q,\dot q)\cdot\varphi(t,q,\dot q),
\end{equation}
as we can see by taking the derivative of the constant of motion~\eqref{veryGeneralConstantAlongMotion} with respect to~$t$. However, this leads to the first integral~\eqref{firstIntegralForTotalDerivative} being identically~0 (see again Examples~\ref{spaceShiftForOscillator} and~\ref{spaceDilationForOscillator} in Section~\ref{EnergyMomentumAngularMomentumSection}). In a few notable instances we will find a different, nontrivial~$\psi$ that works for all Lagrangian motions, using the Lagrange equation to make the right replacements in the formula for $\partial_\varepsilon L$.

%%%%%%%%%%%%%%%%%%%%
\section{Time change}
\label{GaugeSpaceAndTimeSection}

Noether's original 1918 paper~\cite{NoetherOriginal} did not use the concept of what we call BH-function, but in its place had a change of independent variables. In the present context this means a change of time. The role of the time change is usually presented in terms of vector fields. Our purpose here is to bypass such geometric theory entirely and rephrase it as a byproduct of changes of variable in the integral Hamiltonian action, the one that is assumed stationary in the Calculus of Variations.

To our surprise we have found two distinct ways of changing the independent variable in the Hamiltonian action integral: the one we present in Subsection~\ref{invariancesWithSpaceAndTimeChangeAndGaugeSubsection} is simpler to describe and it leads to somewhat simpler formulas for the constants of motion; the other one in Subsection~\ref{alternativeApproachToTimeSubsection} is more convoluted, because its definition requires both the time change and its inverse, but it leads to the very same formulas for the first integrals as we get from the standard theory. The two approaches are in fact perfectly equivalent, as we show in Theorem~\ref{equivalenceOfAlternativeInvariance}, in the sense that any constant of motion that we obtain one way can be obtained also in the other, using a modified time change.

At the risk of annoying the reader that is familiar with the formulas arising from the vector field treatment, we have decided to give prominence to the new ``nonstandard'' version of time change in Subsections~\ref{invariancesWithSpaceAndTimeChangeAndGaugeSubsection} and~\ref{EquivalenceOfThreeVersionsSubsection} and in most examples starting from the next Section~\ref{EnergyMomentumAngularMomentumSection}. The more standard-looking approach will be described in Subsection~\ref{alternativeApproachToTimeSubsection}.

We can isolate a bare definition of time change that is shared among the two approaches:

\begin{Definition}\label{timeChangeDefinition}
A \emph{time change} will be a smooth real function $\tau(\varepsilon, t)$, defined for the same times $t$ as the motion~$q(t)$, and for $\varepsilon$ close to~0, with the compatibility condition $\tau(0,t)=t$ for all~$t$. The trivial time change is simply $\tau(\varepsilon,t)=t$ for all~$t$.
\end{Definition}

To help distinguish the two approaches, we will use the notation $\tau (\varepsilon,t)$ for the time change in the first approach, and use $\theta_\varepsilon(t)$ for the second. This choice is a bit less arbitrary than it may seem, because we will use the inverse~$\theta _\varepsilon^{-1}$, whereas we will find no need to inverte the function $t\mapsto \tau(\varepsilon,t)$.

%%%%%%%%%%%%%%%%%
\subsection{(BH-)invariances under space and time change: first approach}
\label{invariancesWithSpaceAndTimeChangeAndGaugeSubsection}

Our starting idea is to take integral of the Hamiltonian action not on the fixed interval $[a,b]$, but on the nonconstant interval $[\tau(\varepsilon,a), \tau( \varepsilon,b)]$.

\begin{Definition}\label{invarianceUnderSpaceAndTimeChangeDef}
We will say that there is infinitesimal \emph{invariance under space and time change} if
\begin{equation}\label{invarianceSpaceAndTimeAction}
  \frac{\partial
  A_{\tau(\varepsilon,a),\tau(\varepsilon,b)}(q_\varepsilon)}
  {\partial\varepsilon}\Big|_{\varepsilon=0}=0
  \qquad\forall a,b\,,
\end{equation}
or, in terms of the Lagrangian,
\begin{equation}\label{invarianceSpaceAndTimeLagrangian}
  \frac{\partial}{\partial\varepsilon}\Bigl(
  L\bigl(\tau(\varepsilon,t),
  q_\varepsilon(\tau(\varepsilon,t)),
  \dot q_\varepsilon(\tau(\varepsilon,t))\bigr)
  \partial_t\tau(\varepsilon,t)
  \Bigr)
  \Big|_{\varepsilon=0}=0
  \quad\forall t\,.
\end{equation}
\end{Definition}

To clarify why equations~\eqref{invarianceSpaceAndTimeAction} and~\eqref{invarianceSpaceAndTimeLagrangian} are equivalent, start with the action
\begin{equation}\label{actionOnVariableInterval}
  A_{\tau(\varepsilon,a),\tau(\varepsilon,b)}(q_\varepsilon)=
  \int_{\tau(\varepsilon,a)}^{\tau(\varepsilon,b)}
  L\bigl(\xi,q_\varepsilon(\xi),
  \dot q_\varepsilon(\xi)\bigr)\,d\xi
\end{equation}
and perform the change of variable $\xi=\tau(\varepsilon,t)$, keeping in mind the compatibility condition $\tau(0,t)=t$:
\begin{equation*}
  A_{\tau(\varepsilon,a),\tau(\varepsilon,b)}(q_\varepsilon)%=\\
  =\int_a^b
  L\bigl(\tau(\varepsilon,t),
  q_\varepsilon(\tau(\varepsilon,t)),
  \dot q_\varepsilon(\tau(\varepsilon,t))\bigr)
  \partial_t\tau(\varepsilon,t)\,dt\,.
\end{equation*}
If we take the derivative under the integral sign of this expression with respect to~$\varepsilon$ at $\varepsilon=0$, we see that it vanishes for all~$a,b$ if and only if equation~\eqref{invarianceSpaceAndTimeLagrangian} holds.

\begin{Theorem}[Noether-like theorem with space and time change]\label{originalNoetherWithSpaceAndTime}
Suppose that there is infinitesimal invariance under space and time change, and also that $t\mapsto q(t)$ is a solution to the Lagrange equation. Then the following function is constant:
\begin{equation}\label{constantAlongMotionForTimeChange}
  t\mapsto \partial_{\dot q}L\bigl(t,q(t),\dot q(t)\bigr)
  \cdot\partial_\varepsilon
  q_\varepsilon(t)|_{\varepsilon=0}%+\\
  +L\bigl(t,q(t),\dot q(t)\bigr)
  \partial_\varepsilon\tau(\varepsilon,t)|_{\varepsilon=0}\,.
\end{equation}
\end{Theorem}

The proof of this Theorem will be a special case of the next Theorem~\ref{fullNoetherWithSpaceAndTimeAndGauge}.

While space change alone (Theorem~\ref{simplestTheoremWithSpaceChangeOnly}) was not powerful enough to deduce conservation of energy for autonomous Lagrangian systems, space and time change together are. Consider again a Lagrangian function $L$ that does not depend on~$t$, and take the same space change $q_\varepsilon(t):= q(t+\varepsilon)$ as before, introduce the time change
\begin{equation}\label{timeChangeForEnergy}
  \tau(\varepsilon,t) :=t- \varepsilon\,.
\end{equation}
It is trivial to verify that $A_{\tau(\varepsilon,a), \tau( \varepsilon, b)}(q_\varepsilon)$ does not depend on~$\varepsilon$. This is a neat instance of \emph{finite invariance with space and time change}. The induced first integral is the energy, again. Conservation of energy is the classical prototype of a conservation law that can be obtained in two different ways: (1)~with BH-function but no time change, and (2)~with time change but no BH-function. After Theorem~\ref{EquivalenceOfThreeVersionsTheorem} we will be able to exhibit a host of new examples.

It was Bessel-Hagen~\cite{BesselHagen} in 1921 who added the concept invariance up to a divergence (here called BH-invariance), on top of Noether's invariance under space and time change:

\begin{Definition}\label{gaugeInvarianceUnderSpaceAndTimeChangeDef}
We will say that there is infinitesimal \emph{BH-invariance under space and time change} if
\begin{equation}\label{gaugeInvarianceSpaceAndTimeAction}
  \frac{\partial}
  {\partial\varepsilon}
  \Bigl(A_{\tau(\varepsilon,a),
  \tau(\varepsilon,b)}(q_\varepsilon)
  +G(\varepsilon,b)-
  G(\varepsilon,a)\Bigr)
  \Big|_{\varepsilon=0}=0
\end{equation}
for all $a,b$, or, in terms of the Lagrangian,
\begin{equation}\label{gaugeInvarianceSpaceAndTimeLagrangian}
  \frac{\partial}{\partial\varepsilon}\Bigl(
  L\bigl(\tau(\varepsilon,t),
  q_\varepsilon(\tau(\varepsilon,t)),
  \dot q_\varepsilon(\tau(\varepsilon,t))\bigr)
  \partial_t\tau(\varepsilon,t)
  +\partial_tG(\varepsilon,t)
  \Bigr)
  \Big|_{\varepsilon=0}=0
  \quad\forall t\,.
\end{equation}
\end{Definition}

The equivalence of equations~\eqref{gaugeInvarianceSpaceAndTimeAction} and~\eqref{gaugeInvarianceSpaceAndTimeLagrangian} is left to the reader to check.

\begin{Theorem}[Full Noether-like theorem with space and time change and BH-function]\label{fullNoetherWithSpaceAndTimeAndGauge}
Suppose that there is infinitesimal BH-invariance under space and time change, and that also $t\mapsto q(t)$ is a solution to the Lagrange equation. Then the following function is constant in~$t$:
\begin{equation}
  \label{constantAlongMotionForSpaceTimeChangeAndGauge}
  \begin{split}
  N(t)={}& \partial_{\dot q}L\bigl(t,q(t),\dot q(t)\bigr)
  \cdot\partial_\varepsilon
  q_\varepsilon(t)|_{\varepsilon=0}+{}\\
  &{}+L\bigl(t,q(t),\dot q(t)\bigr)
  \partial_\varepsilon\tau(\varepsilon,t)|_{\varepsilon=0}
  +\partial_\varepsilon
  G(\varepsilon,t)
  |_{\varepsilon=0}\,.
  \end{split}
\end{equation}
\end{Theorem}

\begin{proof}
We can compute the derivative in formula~\eqref{gaugeInvarianceSpaceAndTimeAction} using differentiation under the integral sign and the compatibility condition $\tau(0,t)=t$:
\begin{align}
  0={}&\frac{\partial}
  {\partial\varepsilon}
  \Bigl(\int_{\tau(\varepsilon,a)}^{\tau(\varepsilon,b)}
  L\bigl(t,q_\varepsilon(t),\dot q_\varepsilon(t)
  \bigr)dt+G(\varepsilon,b)-
  G(\varepsilon,a)\Bigr)
  \bigg|_{\varepsilon=0}=\notag\\
  ={}&\frac{\partial}
  {\partial\varepsilon}
  \int_{a}^{b}L\bigl(t,q_\varepsilon(t),\dot q_\varepsilon(t)
  \bigr)dt+
  \label{derivativeOfIntegral}
  L\bigl(b,q(b),\dot q(b)\bigr)\partial_\varepsilon\tau(0,b)-{}\\
  &-L\bigl(a,q(a),\dot q(a)\bigr)\partial_\varepsilon\tau(0,a)
  +\partial_\varepsilon G(0,b)-
  \partial_\varepsilon G(0,a)\,.\notag
\end{align}
To the integral term~\eqref{derivativeOfIntegral} we can apply Lemma~\ref{lemmaOfIntegralFreeFormula}. Collecting the terms in~$b$ and in~$a$ we obtain $0=N(b)-N(a)$ for all~$a,b$. This means that $N$ is constant.
\end{proof}

We leave it to the reader to find an generalization of the Definition~\ref{definitionOfTotalDerivativeCondition} of the total derivative condition that is appropriate to Theorem~\ref{fullNoetherWithSpaceAndTimeAndGauge}, as is done by Desloge and Karch~\cite{DeslogeKarch}. In Section~\ref{PlaneWaveLikeExample} we will show a conservation law that is most simply treated with BH-invariance with both space and time change, all of them nontrivial.

%%%%%%%%%%%%%%%%
\subsection{Equivalence of the infinitesimal invariances}
\label{EquivalenceOfThreeVersionsSubsection}

Infinitesimal BH-invariance under both space and time change clearly inglobates the other invariances, where either the gauge or the time change are absent (or trivial: $\tau(\varepsilon,t)\equiv t$, $G\equiv0$). After Bessel-Hagen's 1921 work~\cite{BesselHagen}, it apparently took until Boyer~\cite{Boyer} in 1967 to notice that, given a system whose first integrals can be deduced with space and time change and BH-function, those same integrals can be obtained with space change and a modified BH-function. Our contribution in the next Theorem~\ref{EquivalenceOfThreeVersionsTheorem} is that those first integrals can also be derived with space change and a modified time change, with no BH-function, at least in our ODE setting. In Sections~\ref{EnergyMomentumAngularMomentumSection} onward the only systems to which we cannot apply the Equivalence Theorem~\ref{EquivalenceOfThreeVersionsTheorem} are the few where space change is enough, as in Theorem~\ref{simplestTheoremWithSpaceChangeOnly}.

\begin{Theorem}[Equivalence of infinitesimal invariance conditions]
\label{EquivalenceOfThreeVersionsTheorem}
Let $q_\varepsilon(t)$ be a space change, $\tau(\varepsilon,t)$ a time change, $G(\varepsilon,t)$ a BH-function.
Define the additional BH-function $\mathcal{G}$ and time change~$\mathcal{T}$:
\begin{align}
  \mathcal{G}(\varepsilon,t):={}&\varepsilon\cdot L\bigl(t,
  q(t),\dot q(t)\bigr)
  \bigl(\partial_\epsilon
  \tau(\epsilon,t)|_{\epsilon=0}\bigr)\,,
  \label{auxiliaryGauge}\\
  \mathcal{T}(\varepsilon,t):={}&\varepsilon\cdot
  \frac{\partial_\epsilon G(\epsilon,t)|_{\epsilon=0}
  }%
  {L\bigl(t,q(t),
  \dot q(t)\bigr)}\,.
  \label{auxiliaryTime}
\end{align}
(If needed, we will restrict the time~$t$ to an interval where the denominator of~$\mathcal{T}$ does not vanish). Then the following three conditions are equivalent:
\begin{enumerate}
\item\label{onlyTimeChangeCondition}
the infinitesimal invariance with space and time change of formu\-la~\eqref{invarianceSpaceAndTimeLagrangian} holds with $\tau$ replaced by the time change $\tau+\mathcal{T}$, i.e.:
\begin{equation}
  \label{onlyTimeChangeConditionEquation}
  \frac{\partial}{\partial\varepsilon}\Bigl(
  L\bigl(\xi,
  q_\varepsilon(\xi),\dot q_\varepsilon(\xi)\bigr)
  \big|_{\xi=\tau(\varepsilon,t)+\mathcal{T}(\varepsilon,t)}
  \bigl(\partial_t\tau(\varepsilon,t)+
  \partial_t\mathcal{T}(\varepsilon,t)
  \bigr)
  \Bigr)\Big|_{\varepsilon=0}\equiv0\,.
\end{equation}

\item\label{onlyGaugeCondition}
the infinitesimal BH-invariance with space change of formula~\eqref{gaugeAndSpace} holds with $G$ replaced by~$G+ \mathcal{G}$, i.e.,
\begin{equation}\label{onlyGaugeConditionEquation}
  \frac{\partial}{\partial\varepsilon}\Bigl(
  L\bigl(t,
  q_\varepsilon(t),\dot q_\varepsilon(t)\bigr)%+{}\\
  +\partial_tG(\varepsilon,t)+
  \partial_t\mathcal{G}(\varepsilon,t)
  \Bigr)
  \Big|_{\varepsilon=0}\equiv0\,;
\end{equation}

\item\label{timeChangeAndGaugeCondition}
the infinitesimal BH-invariance with both space and time change of formula~\eqref{gaugeInvarianceSpaceAndTimeLagrangian} holds for the given time change~$\tau$ and BH-function~$G$, i.e.,
\begin{equation}\label{infinitesimalInvarianceWithTauAndGauge}
  \frac{\partial}{\partial\varepsilon}\Bigl(
  L\bigl(\xi,q_\varepsilon(\xi),\dot q_\varepsilon(\xi)\bigr)
  \Big|_{\xi=\tau(\varepsilon,t)}
  \partial_t\tau(\varepsilon,t)%+{}\\
  +\partial_tG(\varepsilon,t)
  \Bigr)
  \Big|_{\varepsilon=0}\equiv0\,.
\end{equation}
\end{enumerate}
\end{Theorem}

\begin{proof}
The left-hand sides of equations~\eqref{onlyTimeChangeConditionEquation}, \eqref{onlyGaugeConditionEquation} and~\eqref{infinitesimalInvarianceWithTauAndGauge} are identically the same. In terms of $\mathcal{L}(\varepsilon,\xi):=L(\xi, q_\varepsilon (\xi), \dot q_\varepsilon(\xi))$, their common value can be expanded out as:
\begin{equation*}
  \mathcal{L}(0,t)\partial^2_{\varepsilon,t}\tau(0,t)+
  \partial_t\mathcal{L}(0,t)\partial_\varepsilon\tau(0,t)%+{}\\
  +\partial_\varepsilon\mathcal{L}(0,t)+
  \partial^2_{\varepsilon,t}G(0,t)\,.
\end{equation*}
This is a straightforward brute-force computation using basic two-variable chain rule calculus, with a little care due to nesting, and using the simplification rules $\tau(0,t)\equiv t$, $\partial_t \tau(0,t)\equiv 1$.
\end{proof}

\begin{Observation}
One can check that the value of the constant of motion in equation~\eqref{constantAlongMotionForSpaceTimeChangeAndGauge} does not change if we perform either the replacements $\tau\to \tau+ \mathcal{T}$, $G\to0$ of Condition~\ref{onlyTimeChangeCondition}, or the replacements $\tau\to t$, $G\to H+\mathcal{G}$ of Condition~\ref{onlyGaugeCondition}.
\end{Observation}

\begin{Observation}
We may feel uneasy that formula~\eqref{auxiliaryTime} for $\mathcal{T}$ contains the Lagrangian~$L$ at the denominator. What happens if the Lagrangian vanishes for some values of~$t$? Are those values of any intrinsic importance in the trivialized BH-function approach? Luckily the answer is negative, thanks to this simple trick: choose a constant~$k$ so that $L(t,q_0(t),\dot q_0(t))+k$ does not vanish in a compact interval we are interested~in, and define the modified functions:
\begin{equation*}
  L_k=L+k\,,\quad
  \mathcal{T}_k(\varepsilon,t):=\varepsilon
  \frac{\partial_\varepsilon
  G(0,t)-k\partial_\varepsilon\tau(0,t)}{L\bigl(t,q_0(t),
  \dot q_0(t)\bigr)+k}\,.
\end{equation*}
Then we may substitute the following Condition~4 for Condition~\ref{onlyTimeChangeCondition} in Theorem~\ref{EquivalenceOfThreeVersionsTheorem}:
\begin{itemize}
\item[4.]\label{modifiedOnlyTimeChangeCondition}
the infinitesimal BH-invariance with time change of formula~\eqref{invarianceSpaceAndTimeLagrangian} holds with $L$ replaced by~$L_k$, the time change $\tau$ replaced by~$\tau+ \mathcal{T}_k$, i.e.:
\begin{equation*}
  \frac{\partial}{\partial\varepsilon}\Bigl(
  L_k\bigl(\xi,
  q_\varepsilon(\xi),\dot q_\varepsilon(\xi)\bigr)
  \big|_{\xi=\tau(\varepsilon,t)+\mathcal{T}_k(\varepsilon,t)}
  \bigl(\partial_t\tau(\varepsilon,t)+
  \partial_t\mathcal{T}_k(\varepsilon,t)
  \bigr)
  \Bigr)\Big|_{\varepsilon=0}\equiv0\,.
\end{equation*}
\end{itemize}
Of course the Lagrangians $L$ and $L_k$ have the same Lagrange equations. Also, the value of the constant of motion in equation~\eqref{constantAlongMotionForTimeChange} does not change if $L$ is replaced by~$L_k$, and $\tau$ by~$\tau+ \mathcal{T}_k$.
\end{Observation}

%%%%%%%%%%%%%%%%%%%%%%
\subsection{A more standard approach to time change}
\label{alternativeApproachToTimeSubsection}

Invariance under space and time change can be introduced in a different way. Let us take a time change that we write as $\theta_\varepsilon(t)$, so that we have a handy notation for its inverse $\theta_\varepsilon^{-1}$ with respect to~$t$. Let $\tilde q_\varepsilon$ be the composition of $t\mapsto q_\varepsilon(t)$ with $\theta_\varepsilon^{-1}$:
\begin{equation}
  \tilde q_\varepsilon(\xi):=q_\varepsilon\circ
  \theta_\varepsilon^{-1}(\xi)\,,\qquad
  \dot{\tilde q}_\varepsilon(\xi)=
  \frac{\dot q_\varepsilon\circ\theta_\varepsilon^{-1}(\xi)}%
  {\theta_\varepsilon'\circ\theta_\varepsilon^{-1}(\xi)}\,.
\end{equation}
Compared to Definitions~\ref{invarianceUnderSpaceAndTimeChangeDef} and~\ref{gaugeInvarianceUnderSpaceAndTimeChangeDef}, the idea here is to take the action with variable endpoints but on the reparameterized $\tilde q_\varepsilon$, instead of the original $q_\varepsilon$.

\begin{Definition}
\label{alternativeGaugeInvarianceUnderSpaceAndTimeChangeDef}
We will say that there is infinitesimal \emph{alternative BH-invariance under space and time change} if
\begin{equation}\label{alternativeGaugeInvarianceSpaceAndTimeAction}
  \frac{\partial}
  {\partial\varepsilon}
  \Bigl(A_{\theta_\varepsilon(a),\theta_\varepsilon(b)}
  (\tilde q_\varepsilon)
  +G(\varepsilon,b)-
  G(\varepsilon,a)\Bigr)
  \Big|_{\varepsilon=0}=0
\end{equation}
for all~$a,b$, or, in terms of the Lagrangian,
\begin{equation}
  \label{alternativeGaugeInvarianceSpaceAndTimeLagrangian}
  \frac{\partial}{\partial\varepsilon}\biggl(
  L\Bigl(\theta_\varepsilon(t),q_\varepsilon(t),
  \frac{\dot q_\varepsilon(t)}{\theta'_\varepsilon(t)}
  \Bigr)\theta'_\varepsilon(t)+
  \partial_t G(\varepsilon,t)
  \biggr)\bigg|_{\varepsilon=0}=0
  \quad\forall t\,.
\end{equation}
\end{Definition}

Formula~\eqref{alternativeGaugeInvarianceSpaceAndTimeLagrangian} can be deduced from~\eqref{alternativeGaugeInvarianceSpaceAndTimeAction} using the change of variable $\xi=\theta_\varepsilon(t)$ to get an integral with fixed endpoints:
\begin{equation*}
  \begin{split}
  A_{\theta_\varepsilon(a),\theta_\varepsilon(b)}
  (\tilde q_\varepsilon)={}&
  \int_{\theta_\varepsilon(a)}^{\theta_\varepsilon(b)}
  L\bigl(\xi,\tilde q_\varepsilon(\xi),
  \dot{\tilde q}_\varepsilon(\xi)\bigr)d\xi=\\
  ={}&
  \int_a^bL\Bigl(\theta_\varepsilon(t),q_\varepsilon(t),
  \frac{\dot q_\varepsilon(t)}{\theta'_\varepsilon(t)}
  \Bigr)\theta'_\varepsilon(t)dt\,.
  \end{split}
\end{equation*}

Condition~\eqref{alternativeGaugeInvarianceSpaceAndTimeLagrangian} is different from~\eqref{gaugeInvarianceSpaceAndTimeLagrangian}. 
The two approaches to invariance can be reconciled with suitable modification of the space changes.

\begin{Theorem}\label{equivalenceOfAlternativeInvariance}
Suppose that a triple of space and time change and BH-function $q_\varepsilon (t), \tau(\varepsilon, t),G(\varepsilon,t)$ satisfies the invariance of Definition~\ref{gaugeInvarianceUnderSpaceAndTimeChangeDef}. Define the new space change~$Q_{1,\varepsilon}(t)$ in either of the following two ways:
\begin{align}
  Q_{1,\varepsilon}(t):={}&q_\varepsilon\bigl(t+
  \varepsilon\cdot(\partial_\epsilon
  \tau(\epsilon,t)|_{\epsilon=0})\bigr)\,,
  \label{fromStandardToAlternative1}\\
  Q_{1,\varepsilon}(t):={}&
  q_\varepsilon(t)+\varepsilon\cdot(\partial_\epsilon
  \tau(\epsilon,t)|_{\epsilon=0})\dot q(t)
  \label{fromStandardToAlternative2}
\end{align}
and simply set $\theta_\varepsilon(t):=\tau(\varepsilon,t)$. Then the triple $Q_{1,\varepsilon}(t),\theta_\varepsilon(t),G(\varepsilon,t)$ satisfies the alternative invariance of Definition~\ref{alternativeGaugeInvarianceUnderSpaceAndTimeChangeDef}.

Conversely, suppose that a triple of space and time change and BH-function $q_\varepsilon (t), \theta_\varepsilon(t), G(\varepsilon, t)$ satisfies the alternative invariance of Definition~\ref{alternativeGaugeInvarianceUnderSpaceAndTimeChangeDef}. Define the new space change~$Q_{2,\varepsilon}(t)$ in either of the following two ways:
\begin{align}
  Q_{2,\varepsilon}(t):={}&q_\varepsilon\bigl(t-
  \varepsilon\cdot(\partial_\epsilon
  \theta_\varepsilon(t)|_{\epsilon=0})\bigr)\,,
  \label{fromAlternativeToStandard1}\\
  Q_{2,\varepsilon}(t):={}&
  q_\varepsilon(t)-\varepsilon\cdot(\partial_\epsilon
  \theta_\varepsilon(t)|_{\epsilon=0})\dot q(t)
  \label{fromAlternativeToStandard2}
\end{align}
and simply set $\tau(\varepsilon,t):=\theta_\varepsilon(t)$. Then the triple $Q_{2,\varepsilon}(t),\tau(\varepsilon,t),G(\varepsilon,t)$ satisfies the invariance of Definition~\ref{gaugeInvarianceUnderSpaceAndTimeChangeDef}.
\end{Theorem}

\begin{proof}
Simply notice that
\begin{gather*}
  \frac{\partial}{\partial\varepsilon}
  A_{\tau(\varepsilon,a),\tau(\varepsilon,b)}(q_\varepsilon)
  \Big|_{\varepsilon=0}=
  \frac{\partial}{\partial\varepsilon}
  A_{\theta_\varepsilon(a),\theta_\varepsilon(b)}
  (\tilde Q_{1,\varepsilon})
  \Big|_{\varepsilon=0}\,,\\
  \frac{\partial}{\partial\varepsilon}
  A_{\theta_\varepsilon(a),\theta_\varepsilon(b)}
  (\tilde q_\varepsilon)
  \Big|_{\varepsilon=0}=
  \frac{\partial}{\partial\varepsilon}
  A_{\tau(\varepsilon,a),\tau(\varepsilon,b)}(Q_{2,\varepsilon})
  \Big|_{\varepsilon=0}\,,
\end{gather*}
where $\tilde Q_{1,\varepsilon}:= Q_{1,\varepsilon}\circ \theta_ \varepsilon ^{-1}$, and $\tilde q_\varepsilon:=q_\varepsilon\circ \theta_ \varepsilon^{-1}$.
\end{proof}

\begin{Theorem}[Full Noether's theorem with space and time change and BH-function]
\label{alternativeFullNoetherWithSpaceAndTimeAndGauge}
Suppose that there is infinitesimal alternative BH-invariance under space and time change according to Definition~\ref{alternativeGaugeInvarianceUnderSpaceAndTimeChangeDef}, and that also $t\mapsto q(t)$ is a solution to the Lagrange equation. Then the following function is constant in~$t$:
\begin{equation}
  \label{constantAlongMotionForAlternativeSpaceTimeChangeAndGauge}
  \begin{split}
  N(t)={}& \partial_{\dot q}L\bigl(t,q(t),\dot q(t)\bigr)\cdot
  \Bigl(\partial_\varepsilon q_\varepsilon(t)|_{\varepsilon=0}-
  (\partial_\varepsilon\theta_\varepsilon(t)|_{\varepsilon=0})
  \dot q(t)\Bigr)+{}\\
  &+L\bigl(t,q(t),\dot q(t)\bigr)
  \partial_\varepsilon\theta_\varepsilon(t)|_{\varepsilon=0}+
  \partial_\varepsilon
  G(\varepsilon,t)
  |_{\varepsilon=0}\,.
  \end{split}
\end{equation}
\end{Theorem}

\begin{proof}
Simply apply Theorem~\ref{fullNoetherWithSpaceAndTimeAndGauge} to the triple $Q_{2,\varepsilon}(t),\tau(\varepsilon,t), G(\varepsilon,t)$ defined in the second half of Theorem~\ref{equivalenceOfAlternativeInvariance}.
\end{proof}

In terms of the alternative invariances with Definition~\ref{alternativeGaugeInvarianceUnderSpaceAndTimeChangeDef}, the equivalence Theorem~\ref{EquivalenceOfThreeVersionsTheorem} can be rephrased this way:

\begin{Theorem}[Equivalence of invariance conditions in terms of~$\theta_\varepsilon$]
\label{EquivalenceOfThreeVersionsUsingInverseOfThetaTheorem}
Let $q_\varepsilon(t)$ be a space change, $\theta_\varepsilon(t)$ a time change, $G(\varepsilon,t)$ a BH-function.
Define the additional BH-function $\mathcal{G}$ and time change~$\mathcal{T}$:
\begin{equation*}
  \mathcal{G}(\varepsilon,t):=\varepsilon\cdot L\bigl(t,
  q(t),\dot q(t)\bigr)
  \bigl(\partial_\varepsilon
  \theta_\varepsilon(t)|_{\varepsilon=0}\bigr)\,,\quad
  \mathcal{T}(\varepsilon,t):=\varepsilon\cdot
  \frac{\partial_\varepsilon G(\varepsilon,t)|_{\varepsilon=0}
  }%
  {L\bigl(t,q(t),
  \dot q(t)\bigr)}\,.
\end{equation*}
Then the following three conditions are equivalent:
\begin{enumerate}
\item\label{onlyTimeChangeConditionWithInverseOfTheta}
the infinitesimal invariance with space and time change of formu\-la~\eqref{alternativeGaugeInvarianceSpaceAndTimeAction} holds with $G$ replaced by the constant~0, the time change $\theta_\varepsilon(t)$ replaced by $\theta_\varepsilon(t)+\mathcal{T}(\varepsilon,t)$, and $q_\varepsilon(t)$ replaced by either of the following
\begin{align*}
  \mathcal{Q}_\varepsilon(t):={}&
  q_\varepsilon\Bigl(t+
  \varepsilon\cdot
  \bigl(\partial_\epsilon\mathcal{T}(\epsilon,t)
  |_{\epsilon=0}\bigr)\Bigr)\,,\\
  \mathcal{Q}_\varepsilon(t):={}&
  q_\varepsilon(t)+\varepsilon\cdot
  \bigl(\partial_\epsilon\mathcal{T}(\epsilon,t)
  |_{\epsilon=0}\bigr)\dot q(t)
\end{align*}
\item\label{onlyGaugeConditionWithInverseOfTheta}
the infinitesimal BH-invariance with space change of formula~\eqref{gaugeAndSpace} holds with $G$ replaced by~$G+ \mathcal{G}$ and $q_\varepsilon(t)$ replaced by either of the following
\begin{align*}
  \mathcal{Q}_\varepsilon(t):={}&
  q_\varepsilon\Bigl(t-
  \varepsilon\cdot
  \bigl(\partial_\epsilon\theta_\epsilon(t)
  |_{\epsilon=0}\bigr)\Bigr)\,,\\
  \mathcal{Q}_\varepsilon(t):={}&
  q_\varepsilon(t)-\varepsilon\cdot
  \bigl(\partial_\epsilon\theta_\epsilon(t)
  |_{\epsilon=0}\bigr)\dot q(t)
\end{align*}
\item\label{timeChangeAndGaugeConditionWithInverseOfTheta}
the infinitesimal alternative BH-invariance with both space and time chan\-ge of formula~\eqref{alternativeGaugeInvarianceSpaceAndTimeAction} holds for the given space change~$q_\varepsilon$, time change~$\theta_\varepsilon(t)$ and BH-function~$G$.
\end{enumerate}
\end{Theorem}

You will notice that, in this formulation, trivializing either time change or BH-function involves modifying the space change too. We wonder if we would have discovered the equivalence results of Theorem~\ref{EquivalenceOfThreeVersionsTheorem} if we had worked with the standard time change~$\theta_\varepsilon$ and had ignored~$\tau_\varepsilon$.  Also, formula~\eqref{constantAlongMotionForAlternativeSpaceTimeChangeAndGauge} for the constant of motion for the alternative invariance is more complicated than formula~\eqref{constantAlongMotionForSpaceTimeChangeAndGauge}. The very definition of alternative invariance involves an inverse time change $\theta_\varepsilon^{-1}$, and as such is somewhat less elementary than the original Definition~\ref{gaugeInvarianceUnderSpaceAndTimeChangeDef}. These are the reasons why we have decided to standardize time change to Definitions~\ref{invarianceUnderSpaceAndTimeChangeDef} and~\ref{gaugeInvarianceUnderSpaceAndTimeChangeDef} throughout most of the rest of this paper, although Definition~\ref{alternativeGaugeInvarianceUnderSpaceAndTimeChangeDef} seems to be the standard in the literature.

%%%%%%%%%%%%%%%%%%%%%%%%%%%%%%%%%
\section{The simplest examples}
\label{EnergyMomentumAngularMomentumSection}

\begin{Example}\label{timeShiftInvarianceExample}
Let us see how Theorem~\ref{EquivalenceOfThreeVersionsTheorem} works out for Conservation of Energy when the Lagrangian $L(q,\dot q)$ is autonomous. As already noted in Section~\ref{infinitesimalVariationalInvarianceSection}, there is infinitesimal BH-invariance with space change type with the choices
\begin{equation}\label{conservationOfEnergyWithGauge}
  q_\varepsilon(t):=q(t+\varepsilon)\,,\quad
  \tau_1(\varepsilon,t):=t\,,\quad
  G_1(\varepsilon,t):=-\varepsilon L\bigl(q(t),\dot q(t)\bigr)\,.
\end{equation}
Formulas~\eqref{auxiliaryGauge} and~\eqref{auxiliaryTime} then become
\begin{equation*}
  \mathcal{G}_1\equiv0\,,\quad
  \mathcal{T}_1(\varepsilon,t):=\varepsilon
  \frac{\partial_\varepsilon G(0,t)}{L\bigl(q(t),
  \dot q(t)\bigr)}=
  \varepsilon\frac{-L\bigl(q(t),\dot q(t)\bigr)}{L\bigl(q(t),
  \dot q(t)\bigr)}=-\varepsilon\,,
\end{equation*}
Theorem~\ref{EquivalenceOfThreeVersionsTheorem} says that there is infinitesimal invariance with space and time change, that is, with the alternative choices
\begin{equation}\label{conservationOfEnergyWithTimeChange}
  q_\varepsilon(t):=q(t+\varepsilon)\,,\quad
  \tau_2(\varepsilon,t):=
  \tau_1(\varepsilon,t)+\mathcal{T}_1(\varepsilon,t)=
  t-\varepsilon\,,\quad
  G_2\equiv0\,.
\end{equation}
We have rediscovered precisely the time change $t-\varepsilon$ of formula~\eqref{timeChangeForEnergy} in Subsection~\ref{invariancesWithSpaceAndTimeChangeAndGaugeSubsection}.

If we had started out with these last choices~\eqref{conservationOfEnergyWithTimeChange}, we would get
\begin{equation*}
  \mathcal{G}_2(\varepsilon,t):=
  \varepsilon L\bigl(q(t),\dot q(t)\bigr)
  \bigl(\partial_\varepsilon\tau_2(\varepsilon,t)\big|_{\varepsilon=0}
  \bigr)
  =-\varepsilon L\bigl(q(t),\dot q(t)\bigr)\,,\quad
  \mathcal{T}_2:=0\,,
\end{equation*}
and we would be led back to infinitesimal BH-invariance with space change type with the original choice~\eqref{gaugeForEnergy} or~\eqref{conservationOfEnergyWithGauge}. Whichever the approach, in the end the energy first integral of equation~\eqref{constantAlongMotionForSpaceTimeChangeAndGauge} becomes $\partial_{\dot q}L(q,\dot q)\cdot\dot q-L(q,\dot q)$. It is unpleasant that the invariance with BH-function given  by~\eqref{conservationOfEnergyWithGauge} is merely infinitesimal, while the invariance with nontrivial time change given by~\eqref{conservationOfEnergyWithTimeChange} is actually a finite invariance. We propose here the following alternative choice for~$G$
\begin{equation}\label{conservationOfEnergyWithGaugeAlternative}
  q_\varepsilon(t):=q(t+\varepsilon)\,,\quad
  \tau_2(\varepsilon,t):=t\,,\quad
  G_2(\varepsilon,t):=-\int_t^{t+\varepsilon}
  L\bigl(q(\xi),\dot q(\xi)\bigr)d\xi
\end{equation}
which recovers a perfect finite invariance. In fact, the quantity
\begin{align}\label{conservationOfEnergyFiniteInvarianceWithGauge}
  f_{a,b}(\varepsilon):={}&
  A_{\tau_2(\varepsilon,a),\tau_2(\varepsilon,b)}(q_\varepsilon)
  +G_2(\varepsilon,b)-
  G_2(\varepsilon,a)=\\
 ={}&
  \int_{a}^{b}
  L\bigl(q(\xi+\varepsilon),\dot q(\xi+\varepsilon)\bigr)\,d\xi
  +G_2(\varepsilon,b)-
  G_2(\varepsilon,a)=\notag\\
  ={}&
  \int_{a}^{b}
  L\bigl(q(t),\dot q(t)\bigr)\,dt
  \,,
\end{align}
does not depend on $\varepsilon$. Or again, if your favourite mnemonic reference is the non-integrated formula~\eqref{infinitesimalInvarianceWithTauAndGauge}, the expression
\begin{equation*}
  L\bigl(q_\varepsilon(\xi),\dot q_\varepsilon(\xi)\bigr)
  \Big|_{\xi=\tau_1(\varepsilon,t)}
  \partial_t\tau_1(\varepsilon,t)+
  \partial_tG_2(\varepsilon,t)
  = L\bigl(q(t),\dot q(t)\bigr)
\end{equation*}
does not depend on $\varepsilon$. The associated first integral is again the energy. The BH-function choice~\eqref{conservationOfEnergyWithGaugeAlternative} seems to be unusual, because it is an \emph{integral functional}, and not a point function of $q(t),\dot q(t)$.

We could apply again Theorem~\ref{EquivalenceOfThreeVersionsTheorem} to the triple $q_\varepsilon,\tau_2,G_2$ and get a new variant with trivial BH-function and nontrivial time change, but we would relinquish the finiteness of the invariance.

The reader can try the alternative space change $q_\varepsilon(t):= q(t)+\varepsilon\cdot\dot q(t)$, which is simply the first-order expansion of $q(t+\varepsilon)$ with respect to~$\varepsilon$.
\end{Example}

\begin{Example}\label{translationInvarianceExample}
(Momentum).
Suppose that $L(t,q,\dot q)$ is invariant in the direction of~$u\in\R^n$: $L(t,q+\varepsilon u,\dot q)\equiv L(t,q,\dot q)$
for all $t,\varepsilon\in\R$, $q,\dot q\in\R^n$. Then for any $q(t)$ there is obvious finite invariance for the translated family 
$q_\varepsilon (t):=q(t)+\varepsilon u$, $\tau(\varepsilon,t)\equiv t$, $G\equiv0$. The constant of motion when $q(t)$ solves the Lagrange equations is the component of the momentum in the direction of~$u$:
\begin{equation*}
  \partial_{\dot q}L\bigl(t,q(t),\dot q(t)\bigr)
  \cdot u\,.
\end{equation*}
Of course here $\mathcal{G}=\mathcal{T}\equiv0$, and all three conditions of Theorem~\ref{EquivalenceOfThreeVersionsTheorem} collapse into one.
\end{Example}

\begin{Example}\label{rotationalInvarianceExample}
(Angular momentum)
Consider a point in the $\R^2$ plane that is driven by a (possibly time-dependent) central force field: $L(t,q,\dot q):=\frac12\lVert\dot q\rVert^2- U\bigl(t,\lVert q\rVert\bigr)$. Given a smooth trajectory $q(t)$ in~$\R^2$, define the rotation family
\begin{equation*}
  q_\varepsilon(t):=\begin{pmatrix}
  \cos\varepsilon & -\sin\varepsilon\\
  \sin\varepsilon & \cos\varepsilon
  \end{pmatrix} q(t)\,,
\end{equation*}
with trivial $\tau(\varepsilon,t):=t$ and $G\equiv0$. It is clear that we have finite invariance: $L\bigl(t,q_\varepsilon(t),\dot q_\varepsilon(t)\bigr)$ does not depend on~$\varepsilon$. Noether's theorem gives the first integral of angular momentum for all Lagrange motions
\begin{equation*}
  \partial_{\dot q} L\cdot
  \partial_\varepsilon q_\varepsilon|_{\varepsilon=0}=
  \dot q\cdot \begin{pmatrix} 0 & -1\\ 1 & 0 \end{pmatrix}q
  =\det(q,\dot q).
\end{equation*}
Again $\mathcal{G}\equiv\mathcal{T}\equiv0$.
\end{Example}

%%%%%%%%%%%%%%%%%%%%%%%%%%
\section{The harmonic oscillator}
\label{harmonicOscillatorSection}

Let us try some ``random'' space changes with the harmonic oscillator.

\begin{Example}\label{spaceShiftForOscillator}
(Space shift for the harmonic oscillator).
Take the one-dimensional harmonic oscillator $L(t,q,\dot q)=\dot q^2/2-q^2/2$ and the space change $q_\varepsilon(t):= q(t)+ \varepsilon$. The Lagrange equation is $\ddot q+q=0$. The constant of motion of Theorem~\ref{NoetherWithSpaceChangeAndGauge} is computed as
\begin{align}
  t\mapsto{}&
  \partial_{\dot q}
  L\bigl(t,q(t),\dot q(t)\bigr)\cdot
  \partial_\varepsilon q_\varepsilon(t)
  \big|_{\varepsilon=0}-
  \int_{t_0}^t
  \frac{\partial}{\partial\varepsilon}
  L\bigl(s,q_\varepsilon(s),\dot q_\varepsilon(s)\bigr)
  \Big|_{\varepsilon=0}ds=\\
   ={}&\dot q(t)-\int_{t_0}^t-q(s)ds=
   \dot q(t)-\int_{t_0}^t\ddot q(s)ds=
   \dot q(t)-\dot q(t)+\dot q(t_0)=\dot q(t_0)\,.
\end{align}
Let us try the total derivative condition:
\begin{equation}
  \Bigl(\frac{\partial}{\partial\varepsilon}
  L\bigl(t,q_\varepsilon(t),\dot q_\varepsilon(t)\bigr)
  \Bigr)\Big|_{\varepsilon=0}=-q(t).
\end{equation}
Can $-q(t)$ be the total time derivative of some function $\psi(t, q(t), \dot q(t))$, with $\psi$ being the same for all smooth paths~$q(t)$? Surely not, because otherwise the integral of $-q(t)$ over an interval $[t_0,t_1]$ would only depend on the end values, and not on the path:
$$\int_{t_0}^{t_1}-q(s)ds=
  \int_{t_0}^{t_1}\frac{d}{ds}\psi(s,q(s), \dot q(s))ds=
  \psi(t_1,q(t_1), \dot q(t_1))-
  \psi(t_0,q(t_0), \dot q(t_0)),$$
which is obviously not the case for generic paths.

If we restrict $q(t)$ to be a motion then
\begin{equation}
  \Bigl(\frac{\partial}{\partial\varepsilon}
  L\bigl(t,q_\varepsilon(t),\dot q_\varepsilon(t)\bigr)
  \Bigr)\Big|_{\varepsilon=0}=-q(t)=\ddot q(t),
\end{equation}
so that we can choose $\psi(t,q,\dot q)\equiv\dot q$. Let us apply Theorem~\ref{NoetherTheoremWithTotalDerivative} (see formula~\eqref{trivialPsi2}). The function $\varphi$ is the constant~1:
\begin{equation}
  \partial_\varepsilon q_\varepsilon(t)\big|_{\varepsilon=0}=1=
  \varphi\bigl(t,q(t),\dot q(t)\bigr)
  \qquad\forall t\,.
\end{equation}
The first integral is trivially~0:
\begin{equation}
  (t,q,\dot q)\mapsto
  \partial_{\dot q}L(t,q,\dot q)\cdot\varphi(t,q,\dot q)-
  \psi(t,q,\dot q)=
  \dot q-\dot q\equiv0\,.
\end{equation}
\end{Example}

\begin{Example}\label{spaceDilationForOscillator}
(Space dilation for the harmonic oscillator).
Take again the one-dimensional harmonic oscillator $L(t,q,\dot q)=\dot q^2/2-q^2/2$ and the space change $q_\varepsilon(t):= (1+\varepsilon) q(t)$. We can compute:
\begin{equation}
  \Bigl(\frac{\partial}{\partial\varepsilon}
  L\bigl(t,q_\varepsilon(t),\dot q_\varepsilon(t)\bigr)
  \Bigr)\Big|_{\varepsilon=0}=\dot q(t)^2-q(t)^2=
  2L\bigl(t,q(t),\dot q(t)\bigr),
\end{equation}
whence the integral constant of motion
\begin{equation}
  q(t)\dot q(t)-2\int_{t_0}^t L\bigl(s,q(s),\dot q(s)\bigr)ds.
\end{equation}
The expression $\dot q^2-q^2$ is not a total derivative, as the reader can check by integrating it on different paths with the same endpoints, end velocities etc. However, if we use the Lagrange equation $\ddot q=-q$ in the following way we do get a total derivative:
\begin{equation}
  \Bigl(\frac{\partial}{\partial\varepsilon}
  L\bigl(t,q_\varepsilon(t),\dot q_\varepsilon(t)\bigr)
  \Bigr)\Big|_{\varepsilon=0}=\dot q(t)^2-q(t)^2=
  \dot q(t)^2+q(t)\ddot q(t)=
  \frac{d}{dt}\bigl(q(t)\dot q(t)\bigr),
\end{equation}
but the resulting constant of motion is again trivially~0. The Lagrange equation $\ddot q=-q$ can be used to transform the expression $\dot q^2-q^2$ in multiple ways, some of which may be more useful than others. For example a less subtle one would give us $\dot q^2-\ddot q^2$, which is not a total derivative.
\end{Example}

\begin{Example}\label{nonlocalSpaceChangeForHarmonicOscillator} 
(Nonlocal space change for the harmonic oscillator).
Take once more the one-dimensional harmonic oscillator $L(t,q,\dot q)=\dot q^2/2-q^2/2$ with the nonlocal space change 
\begin{equation}
  q_\varepsilon(t):= q(t)+\varepsilon\int_0^t q(\tau)d\tau.
\end{equation}
We can compute:
\begin{equation}
  \Bigl(\frac{\partial}{\partial\varepsilon}
  L\bigl(t,q_\varepsilon(t),\dot q_\varepsilon(t)\bigr)
  \Bigr)\Big|_{\varepsilon=0}=
  \frac12\frac{d}{dt}\biggl(q(t)^2-
  \Bigl(\int_0^t q(\tau)d\tau\Bigr)^2\biggr),
\end{equation}
whence the following nonlocal constant of motion, for any arbitrary choice of~$t_0$:
\begin{equation}\label{squareIntegralConstantOfMotion}
  \dot q(t)\int_0^t q(\tau)d\tau
  -\frac12q(t)^2+\frac12q(t_0)^2+
  \frac12\biggl(\int_0^t q(\tau)d\tau\biggr)^2
  -\frac12\biggl(\int_0^{t_0} q(\tau)d\tau\biggr)^2.
\end{equation}
If we use the Lagrange equation $\ddot q=-q$ the constant of motion simplifies to
\begin{equation}
  \bigl(\dot q(0)-\dot q(t_0)\bigr)\dot q(t_0),
\end{equation}
which is nontrivial if $t_0$ is not a multiple of~$2\pi$. One more constant of motion results if we remove the two terms with~$t_0$ (which are obvious constants of motion) from~\eqref{squareIntegralConstantOfMotion}: 
\begin{equation}
  \dot q(t)\int_0^t q(\tau)d\tau
  -\frac12q(t)^2+
  \frac12\biggl(\int_0^t q(\tau)d\tau\biggr)^2.
\end{equation}
which, using $\ddot q=-q$, simplifies to
\begin{equation}
  -\frac12\bigl(q(t)^2+\dot q(t)^2-\dot q(0)^2\bigr).
\end{equation}
We have found a roundabout way to prove that the energy $\frac12(q^2+\dot q^2)$ is constant. This will be less surprising if we go back and use the Lagrange equation already in the expression of~$q_\varepsilon$.
\end{Example}

%%%%%%%%%%%%%%%%%%%%%%%%%%
\section{Dissipative systems}
\label{dissipativeSystemsSection}

Let us take up again the ``dissipative'' Lagrangian
$\mathcal{L}(t,q,\dot q):=e^{ht} L(q,\dot q)$ of formula~\eqref{dissipativeSystemLagrangian},
with the space change and BH-function given by~\eqref{dissipativeSystemSpaceChange} and~\eqref{dissipativeSystemGauge} and the unusual constant of motion
\begin{equation*}
  \partial_{\dot q}\mathcal{L}\bigl(t,q(t),\dot q(t)\bigr)
  \cdot\dot q(t)-
  \mathcal{L}\bigl(q(t),\dot q(t)\bigr)%+{}\\
  +h\int_{t_0}^t\mathcal{L}\bigl(s,q(s),\dot q(s)\bigr)ds\,.
\end{equation*}
For example, take $m,k>0$, $h=k/m$, and the familiar Lagrangian $L(q,\dot q)=\frac{1}{2} m\lVert\dot q\rVert^2-U(q)$. The Lagrange equations for $\mathcal{L}$ are $m\ddot q+\nabla U(q)=-k\dot q$, where we see a viscous resistance term $-k\dot q$, which justifies our use of the term ``dissipative''. The ``usual'' energy for the system is of course
\begin{equation*}
  E(q,\dot q)=\frac{1}{2}m\lVert\dot q\rVert^2+U(q)=
  \partial_{\dot q}L(q,\dot q)\cdot\dot q-
  L(q,\dot q)\,,
\end{equation*}
and it is not conserved. Still, if we switch our attention to a ``dissipative energy'' defined as
\begin{equation*}
  \mathcal{E}(t,q,\dot q)=e^{kt/m}E(q,\dot q)=
  \partial_{\dot q}\mathcal{L}(t,q,\dot q)\cdot\dot q(t)-
  \mathcal{L}(t,q,\dot q)\,,
\end{equation*}
we can rewrite the constant of motion as
\begin{equation}\label{energyPlusAction2}
  \mathcal{E}\bigl(t,q(t),\dot q(t)\bigr)+
  h\int_{t_0}^t
  \mathcal{L}\bigl(s,q(s),\dot q(s)\bigr)ds\,.
\end{equation}
That is, the dissipative energy decreases (assuming~$L>0$) over time and its total loss over any time interval is proportional to the dissipative Hamiltonian action in that interval.

A point-function first integral can be recovered in a special case as follows. Again with the familiar Lagrangian $L(q,\dot q)=\frac{1}{2} m\lVert\dot q\rVert^2-U(q)$ we can compute, using the Lagrange equation,
\begin{equation*}\begin{split}
  \frac{d}{dt} \bigl(m e^{kt/m} q\cdot\dot q\bigr)={}&
  e^{kt/m}\bigl(kq\cdot \dot q+m\lVert\dot q\rVert^2-
  q\cdot\nabla U(q)-kq\cdot \dot q\bigr)=\\
  ={}&e^{kt/m}\bigl(m\lVert\dot q\rVert^2-
  q\cdot\nabla U(q)\bigr)\,.
  \end{split}
\end{equation*}
Assume that $U$ is homogeneous of degree~2, i.e., $U(\lambda q) =\lambda^2U(q)$ if $\lambda>0$. Then $q\cdot\nabla U(q)=2U(q)$, so that the integrand in formula~\eqref{energyPlusAction2} becomes a total derivative:
\begin{equation*}
  \frac{d}{dt} \bigl(m e^{kt/m} q\cdot \dot q\bigr)=
  e^{kt/m}\left(m\lVert\dot q\rVert^2-2 U(q)\right)%=\\
  =2\mathcal{L}\bigl(t,q(t),\dot q(t)\bigr)
\end{equation*}
and the conserved quantity~\eqref{energyPlusAction2} is
\begin{equation*}
  e^{kt/m}\Bigl(
  E(q(t),\dot q(t))+
  \frac{1}{2}kq(t)\cdot \dot q(t)\Bigr)-
  \frac{1}{2}e^{kt_0/m}kq(t_0)\cdot \dot q(t_0).
\end{equation*}
The last term containing~$t_0$, as it is an obvious constant of motion, and we can ignore it. We finally obtain a true point-function first integral:
\begin{equation}\label{dissipativeFirstIntegral}
  e^{kt/m}\Bigl(
  E(q,\dot q)+
  \frac{1}{2}kq\cdot \dot q\Bigr)
\end{equation}
(this formula can also be deduced from Theorem~\ref{NoetherTheoremWithTotalDerivative} more directly). Let us see how we can arrive at this same first integral with a trivial BH-function, following Theorem~\ref{EquivalenceOfThreeVersionsTheorem}. The original $G$ of formula~\eqref{dissipativeSystemGauge} in our specialized assumptions is
\begin{align*}
  G(\varepsilon,t):={}&
  -\varepsilon e^{ht} L\bigl(q(t),\dot q(t)\bigr)
  +\varepsilon h\int_{t_0}^te^{hs }
  L\bigl(q(s),\dot q(s)\bigr)ds=\\
  ={}&-\varepsilon \mathcal{L}\bigl(t,q(t),\dot q(t)\bigr)+
  \varepsilon \frac{k}{m}
  \int_{t_0}^t\mathcal{L}\bigl(s,q(s),\dot q(s)\bigr)ds=\\
  ={}&\varepsilon \mathcal{L}\bigl(t,q(t),\dot q(t)\bigr)+
  \varepsilon\frac{k}{2}
   e^{kt/m} q(t)\cdot \dot q(t)-
   \varepsilon\frac{k}{2}
   e^{kt_0/m} q(t_0)\cdot \dot q(t_0)\,.
\end{align*}
Again we ignore the $t_0$ term and take the simplified
\begin{equation*}
  G_1(\varepsilon,t):=
  -\varepsilon \mathcal{L}\bigl(t,q(t),\dot q(t)\bigr)+
  \varepsilon\frac{k}{2}
  e^{kt/m} q(t)\cdot \dot q(t)\,.
\end{equation*}
The additional time change of formula~\eqref{auxiliaryTime} is
\begin{equation*}
  \mathcal{T}(\varepsilon,t):=
  \varepsilon\cdot
  \frac{-\mathcal{L}+\frac{k}{2}
  e^{kt/m} q\cdot \dot q}{\mathcal{L}}=
  -\varepsilon+\varepsilon k\frac{q\cdot\dot q}{2L}
\end{equation*}
Theorem~\ref{EquivalenceOfThreeVersionsTheorem} tells us that we can obtain the first integral~\eqref{dissipativeFirstIntegral} using
\begin{equation}\label{spaceAndTimeChangeForDissipative}
  q_\varepsilon(t):=q(t+\varepsilon)\,,\quad
  \tau(\varepsilon,t):=
  t-\varepsilon+\varepsilon k\frac{q\cdot\dot q}{L}\,,\quad
  G_2(\varepsilon,t)\equiv0\,.
\end{equation}

The one-dimensional case, where $U(q)=cq^2$, is treated by Logan~\cite{Logan} using a different set of time and space change:
\begin{equation}\label{LoganSpaceAndTimeChangeForDissipative}
  Q_\varepsilon(t):=
  \Bigl(1-\frac{\varepsilon k}{2m}\Bigr)q(t-\varepsilon)\,,
  \quad
  \tau_2(\varepsilon,t):=t+\varepsilon\,,\quad
  G_2(\varepsilon,t)\equiv0\,,
\end{equation}
which he deduces by solving the generalized Killing equations. You will notice that the time change in~\eqref{spaceAndTimeChangeForDissipative} depends on~$\dot q$, while the one in~\eqref{LoganSpaceAndTimeChangeForDissipative} does not. The reader can compute how the equivalence Theorem~\ref{EquivalenceOfThreeVersionsTheorem} applies to the triple~\eqref{LoganSpaceAndTimeChangeForDissipative}.

%%%%%%%%%%%%%%%%%%%%%%%%%%%%%%%%%%%%%%%%%%%

\section{The Lane-Emden equation}
\label{LaneEmdenSection}

The Lane-Emden system has the following Lagrangian function and Lagrange equation:
\begin{equation}
  L(t,q,\dot q):=t^2\Bigl(\frac{\dot q^2}{2}-\frac{q^{n+1}}{n+1}
  \Bigr),\qquad
  \ddot q=-q^n-\frac{2}{t}\dot q,
\end{equation}
where $q\in\R$, $n\in\N$. It is used in Astrophysics to model some aspect of star evolution~\cite{Chandrasekhar}. It is known that if $q(t)$ is a solution then also the rescaled $q_\varepsilon(t):= e^\varepsilon q(e^{\varepsilon (n-1) /2}t)$ is also a solution. Let us take this $q_\varepsilon$ as space change. Assuming that $q(t)$ solves Lagrange's equation we obtain:
\begin{multline}
  \frac{\partial}{\partial\varepsilon}
  L\bigl(t,q_\varepsilon(t),\dot q_\varepsilon(t)\bigr)
  \Big|_{\varepsilon=0}=
%  \\=
  \frac{d}{dt}\biggl(\frac{1-n}{2}tL\bigl(t,q(t),\dot q(t)\bigr)-
  \frac{5-n}{4}t^2q(t)\dot q(t)\biggr)+{}\\
  {}+\frac{(n-5)(n-1)}{4(n+1)}t^2q(t)^{n+1}.
\end{multline}
From Theorem~\ref{theoremOfVeryGeneralConstantOfMotion} we deduce the following integral constant of motion:
\begin{multline}
  \partial_{\dot q}
  L\bigl(t,q(t),\dot q(t)\bigr)\cdot
  \partial_\varepsilon q_\varepsilon(t)
  \big|_{\varepsilon=0}
  -\int \frac{\partial}{\partial\varepsilon}
  L\bigl(t,q_\varepsilon(t),\dot q_\varepsilon(t)\bigr)
  \Big|_{\varepsilon=0}dt=\\
  =-\frac{(n-1) t^2 \bigl((n+1) q(t)
   \dot q(t)+(n+1) t \dot q(t)^2+2 t
   q(t)^{n+1}\bigr)}{4 (n+1)}-{}\\
  -\frac{(n-5)(n-1)}{4(n+1)}\int_{t_0}^ts^2q(s)^{n+1}ds.
\end{multline}
This expression vanishes identically when $n=1$, and it gives a point-function first integral in the well-known case $n=5$~\cite{Mach}.

If we try with the different space change $q_\varepsilon(t):= e^\varepsilon q(e^\varepsilon t)$, or, equivalently, with $q_\varepsilon(t):= q(t)+\varepsilon(q(t)+t\dot q(t))$ and again assume that $q(t)$ solves the Lagrange equation, the calculations are a little simpler:
\begin{equation}
  \frac{\partial}{\partial\varepsilon}
  L\bigl(t,q_\varepsilon(t),\dot q_\varepsilon(t)\bigr)
  \Big|_{\varepsilon=0}
  =\frac{d}{dt}\biggl(-\frac{2}{n+1}t^3q(t)^{n+1}\biggr)
  {}+\frac{5-n}{n+1}t^2q(t)^{n+1}.
\end{equation}
and the consequent integral constant of motion is just a multiple of the previous one:
\begin{equation}
  t^2 \Bigl(\frac{2}{n+1} t q^{n+1}+q\dot q+t
   \dot q^2\Bigr)+\frac{n-5}{n+1} \int t^2 q(t)^{n+1} \,dt
\end{equation}
with the advantage that it is nontrivial even when $n=1$.

The integral constant of motion may be of some interest when $n$ is odd and~$>5$, as for example $n=7$:
\begin{equation}
  t^2 \bigl(t q^8+4 t\dot q^2+4 q\dot q\bigr)+
   \int t^2 q(t)^8 \, dt.
\end{equation}
Observe that the integral $\int t^2q(t)^{n+1}dt$ is increasing in~$t$ and the function $(q,\dot q)\mapsto (n+1) q\dot q+(n+1) t \dot q^2+2 t q^{n+1}$ is coercive for large~$t$, so that it is a Lyapunov function for the system.

%%%%%%%%%%%%%%%%%%%%%%%%%%%%%%%%%%%%%%%%%%%

\section{Homogeneous potentials}
\label{homogeneousPotentialsSection}

Consider a Lagrangian of the form $L(t,q,\dot q):=\frac{1}{2}m \lVert\dot q\rVert^2-U(q)$, where the potential $U$ is positively homogeneous of degree~$\alpha$, i.e., $U(\lambda q)=\lambda^\alpha U(q)$ for all~$\lambda>0$. Define the space change as
\begin{equation*}
  q_\varepsilon(t)=
  e^\varepsilon q\bigl(e^{\varepsilon(\alpha /2-1)}t\bigr)\,.
\end{equation*}
We chose this particular~$q_\varepsilon$ because it is well-known that $q_\varepsilon(t)$ is a Lagrange motion whenever $q(t)$~is. We can compute
\begin{align*}
  \frac{\partial}{\partial\varepsilon}
  L\bigl(t,q_\varepsilon(t),\dot q_\varepsilon(t)
  \bigr)\Big|_{\varepsilon=0}={}&
  \alpha L+\Bigl(\frac{\alpha }{2}-1\Bigr)t\dot L=\\
  ={}&\frac{d}{dt}\biggl(\Bigl(\frac{\alpha}{2}-1\Bigr)tL+
  \Bigl(\frac{\alpha}{2}+1\Bigr)\int_{t_0}^tL\,ds\biggr)\,.
\end{align*}
Hence we can take the following BH-function
\begin{equation*}
  G(\varepsilon,t):=-\varepsilon
  \biggl(\Bigl(\frac{\alpha}{2}-1\Bigr)tL+
  \Bigl(\frac{\alpha}{2}+1\Bigr)\int_{t_0}^tL\,ds\biggr)\,
\end{equation*}
and deduce the following constant of motion:
\begin{equation}\label{homogeneousPotentialConstantOfMotion}
  \bigl(\partial_{\dot q}L\cdot\partial_\varepsilon q_\varepsilon
  +\partial_\varepsilon G(\varepsilon,t)
  \bigr)
  |_{\varepsilon=0}=
  mq\cdot\dot q+\Bigl(\frac{\alpha}{2}-1\Bigr)tE-
  \Bigl(\frac{\alpha}{2}+1\Bigr)\int_{t_0}^tL\,ds\,,
\end{equation}
where $E:=\frac12m\lVert\dot q\rVert^2+U(q)$ is the energy, which is conserved too, because the system is autonomous. As you can see, the constant of motion~\eqref{homogeneousPotentialConstantOfMotion} involves the Hamiltonian action and the energy.

The special case $\alpha =-2$ is specially interesting, because we get a very simple time-dependent first integral in the usual sense:
\begin{equation}\label{homogeneousPotentialFistIntegral}
  F=mq\cdot\dot q-2tE\,.
\end{equation}
Because of Theorem~\ref{EquivalenceOfThreeVersionsTheorem} we can obtain this first integral with the following space and time change and null BH-function:
\begin{equation*}
  q_\varepsilon(t)=e^\varepsilon q(e^{-2\varepsilon}t)\,,\quad
  \tau(\varepsilon,t):=t+\mathcal{T}(\varepsilon,t):=
  t+\varepsilon \frac{2tL}{L}=(1+2\varepsilon)t\,.
\end{equation*}
Another equally valid choice for time change is $\tau_1(\varepsilon, t):= e^{2\varepsilon}t$, which matches better with the space change.

Some consequences of the first integral~\eqref{homogeneousPotentialFistIntegral} are easy to draw. We can take the antiderivative with respect to time in the equation $0=mq\cdot\dot q-2tE-F$ and obtain one more time-dependent constant of motion
\begin{equation}\label{homogeneousPotentialOtherFistIntegral}
  F_1=\frac{1}{2}m\lVert q\rVert^2-t^2E-tF\,.
\end{equation}
This means that, for any given orbit, the couple $(t,\lVert q\rVert)$ always lie on a conic in~$\R^2$: on an ellipse if $E<0$, on a hyperbole if~$E>0$, and on a straight line if~$E=0$ and $F\ne0$, and on a single point if~$E=F=0$. We can also solve for~$\lVert q\rVert$:
\begin{equation}\label{homogeneousPotentialRadiusOnTime}
  \lVert q\rVert=\frac{2}{m}\sqrt{t^2E+tF+F_1}\,.
\end{equation}
This formula gives exactly how the distance from the origin depends on time, even though we don't know the shape of the orbit. For example, when $E<0$ the orbit is born and dies at the origin at the instants $(-F\pm\sqrt{F^2-4EF_1})/(2E)$.

If we solve equation~\eqref{homogeneousPotentialOtherFistIntegral} for~$t$ and replace into equation~\eqref{homogeneousPotentialFistIntegral} we obtain a constant of motion that does not involve the current time~$t$, but only $q(t),\dot q(t)$ and the initial data $t_0,q(t_0),\dot q(t_0)$.

The central potential $U(q)= -k/\Vert q\rVert^2$ is a well-known special case, for which Danby~\cite{Danby} derives formula~\eqref{homogeneousPotentialOtherFistIntegral}. Using the angular momentum, the shapes of the orbits can be calculated explicitly and are called Cotes' spirals \cite[p.~83]{Whittaker}. 

Another famous potential that is homogenous of degree $-2$ is Calogero's inverse-square scattering potential:
\begin{equation}\label{calogeroPotential}
  U(q_1,\dots,q_n)=\sum_{1\le j<k\le n}\frac{1}{(q_j-q_k)^2},
\end{equation}
for $q_j\in\R$, $q_j\ne q_k$ when $j\ne k$. The associated system is analytically integrable~\cite{MoserIsospectral}.

%%%%%%%%%%%%%%%%%%%%%%%%%%%%%%%%%%%%%%%%%%%

\section{The Toda lattice}
\label{todaLattice}

The nonperiodic Toda lattice is a system of $n$ particles $q_1,q_2, \dots,q_n \in\R$, such that each successive pair $q_k,q_{k+1}$ repel each other through an exponential potential. The Lagrangian is
\begin{gather}\label{TodaLagrangian}
  L(Q,\dot Q):=\frac12\lVert\dot Q\rVert^2-\mathcal{V}(Q),\\
  \text{where}\quad
  Q=(q_1,\dots,q_n)\in\R^n,\quad
  \mathcal{V}(Q):=\sum_{k=1}^{n-1}e^{q_k-q_{k+1}}.
\end{gather}
The system was proved to be integrable by H\'enon, Flaschka and Manakov (see for example~\cite{MoserToda}). Here we deduce a simple integral constant of motion. Consider the space change $Q_\varepsilon(t):=Q(e^{-\varepsilon}t)+(\varepsilon,2\varepsilon, \dots,n\varepsilon)$. This is chosen so that $\mathcal{V} (Q_\varepsilon(t))=e^{-\varepsilon} \mathcal{V} (Q(e^{-\varepsilon} t))$. The following relation holds:
\begin{equation}\label{TodaTotalDerivative}
  \frac{\partial}{\partial\varepsilon}
  L\bigl(Q_\varepsilon(t),\dot Q_\varepsilon(t)\bigr)
  \Big|_{\varepsilon=0}=
  -\frac12\lVert\dot Q(t)\rVert^2
  -\frac{d}{dt}\Bigl(tL\bigl(Q(t),\dot Q(t)\bigr)\Bigr).
\end{equation}
The constant of motion of Theorem~\ref{theoremOfVeryGeneralConstantOfMotion} with $t_0=0$ takes the following forms:
\begin{align}\label{TodaConstantOfMotion}
  \partial_{\dot q}
  &L\bigl(Q(t),\dot Q(t)\bigr)\cdot
  \partial_\varepsilon Q_\varepsilon(t)
  \big|_{\varepsilon=0}-
  \int_{0}^t
  \frac{\partial}{\partial\varepsilon}
  L\bigl(Q_\varepsilon(s),\dot Q_\varepsilon(s)\bigr)
  \Big|_{\varepsilon=0}ds=\notag\\
  &=-t\lVert\dot Q(t)\rVert^2+\sum_{k=1}^n k\dot q_k(t)+
  tL\bigl(Q(t),\dot Q(t)\bigr)
  +\frac12\int_{0}^t\lVert\dot Q(s)\rVert^2ds=\notag\\
  &=-tE+\sum_{k=1}^n k\dot q_k(t)
  +\frac12\int_{0}^t\lVert\dot Q(s)\rVert^2ds=\notag\\
  &=-tE+\sum_{k=1}^n k\dot q_k(t)
  +\int_{0}^t\Bigl(\frac12\lVert\dot Q(s)\rVert^2+
  \mathcal{V}(Q(s))\Bigr)ds-
  \int_{0}^t\mathcal{V}(Q(s))ds=\notag\\
  &=\sum_{k=1}^n k\dot q_k(t)-
  \int_{0}^t\mathcal{V}(Q(s))ds,
\end{align}
where $E=\frac12\lVert\dot Q\rVert^2+\mathcal{V}(Q)$ is the (constant) energy.

%%%%%%%%%%%%%%%%%%%%%%%%%%%%%%%%%%%%%%%%%%%

\section{Cone potentials}
\label{conePotentials}

Generalizing Example~\ref{spaceShiftForOscillator} of Section~\ref{spaceShiftForOscillator}, take a Lagrangian of the form $L(t,q,\dot q)=\frac{1}{2}\lVert\dot q\rVert-\mathcal{V}(q)$ and the space-shift $q_\varepsilon(t)= q(t)+\varepsilon v$, with $v\in\R^N$ a constant vector. We can compute $\partial_\varepsilon q_\varepsilon(t) |_{\varepsilon =0}=v$ and, assuming that $q(t)$ is a solution to Lagrange equation $\ddot q=-\nabla\mathcal{V}(q)$,
\begin{align*}
  \Bigl(\frac{\partial}{\partial\varepsilon}
  L\bigl(t,q_\varepsilon(t),\dot q_\varepsilon(t)\bigr)
  \Bigr)\Big|_{\varepsilon=0}={}&
  -\Bigl(\frac{\partial}{\partial\varepsilon}
  \mathcal{V}\bigl(q(t)+\varepsilon v\bigr)
  \Bigr)\Big|_{\varepsilon=0}=
  -\nabla\mathcal{V}\bigl(q(t)\bigr)\cdot v=\\
  ={}&
  \ddot q(t)\cdot v=
  \frac{d}{dt}\dot q\cdot v\,.
\end{align*}
The total derivative condition holds with $\varphi(t,q,\dot q):=v$ for all motions, and $\psi(t,q,\dot q):=\dot q\cdot v$, but unfortunately the corresponding first integral is again identically null, as in the previous Example~\ref{spaceShiftForOscillator}:
\begin{equation*}
  \partial_{\dot q}L(t,q,\dot q)\cdot\varphi(t,q,\dot q)-
  \psi(t,q,\dot q)=
  \dot q\cdot v-\dot q\cdot v=0\,.
\end{equation*}
However, the integral constant of motion of Theorem~\ref{theoremOfVeryGeneralConstantOfMotion} is not null:
\begin{align}
  \partial_{\dot q}
  L\bigl(t,q(t),&\dot q(t)\bigr)\cdot
  \partial_\varepsilon q_\varepsilon(t)
  \big|_{\varepsilon=0}
  -\int_{t_0}^t
  \frac{\partial}{\partial\varepsilon}
  L\bigl(s,q_\varepsilon(s),\dot q_\varepsilon(s)\bigr)
  \Big|_{\varepsilon=0}ds=\notag\\
  &=\dot q(t)\cdot v-
  \int_{t_0}^t
  \frac{\partial}{\partial\varepsilon}
  \bigl(- \mathcal{V}(q(s)+\varepsilon v)\bigr)
  \big|_{\varepsilon=0}ds=\label{initialVelocityAsConstant}\\
  &=\dot q(t)\cdot v-
  \int_{t_0}^t\bigl(- \nabla\mathcal{V}(q(s))\cdot v\bigr)ds=\notag\\
  &=\dot q(t)\cdot v-
  \int_{t_0}^t\ddot q(s)\cdot v\,ds=\notag\\
  &=\dot q(t)\cdot v-\bigl(\dot q(t)-\dot q(t_0)\bigr)\cdot v =
  \notag\\
  &=\dot q(t_0)\cdot v.\label{initialVelocityConstant}
\end{align}
There are classes of potentials $\mathcal{V}$ (scattering potentials, or cone potentials), for which all motions have a finite \emph{asymptotic velocity} $\dot q_\infty := \lim_{t \to+\infty}\dot q(t)$. The cone potential class includes the Toda lattice and Calogero's system that we mentioned in Section~\ref{todaLattice} and~\ref{homogeneousPotentialsSection}. For such systems, the formulas~\eqref{initialVelocityAsConstant} make sense also for $t_0=+\infty$, yielding the asymptotic velocity $p_\infty$ as a vector constant of motion. In the early 1990s the authors~\cite{GZscattering} developed a theory of the Liouville-Arnold integrability of cone potentials, using precisely the components of $\dot q_\infty$ as basic constants of motion.

%%%%%%%%%%%%%%%%%%%%%%%%%%%%%%%%%%%%%%%%%%%
\section{Kepler's problem}
\label{sectionRungeLenz}

The Lagrangian function of Kepler's problem with its associated Lagrange equation are
\begin{equation}\label{KeplerEquation}
  L\bigl(q,\dot q\bigr)=\frac{1}{2}
  \lVert\dot q\rVert^2+\frac{k}{\lVert q\rVert}\,,
  \qquad
  \ddot q(t)=-k\frac{q(t)}{\lVert q(t)\rVert^3}\,,
\end{equation}
where  $k>0$ is a parameter. Notice that for all Kepler motions $q$~is parallel to~$\ddot q$. We will assume that $q$ is a vector in~$\R^2$ to simplify some formulas.

As for all autonomous Lagrangians, the energy is conserved. The Lagrangian is also invariant under rotations around the origin. Noether's theorem gives the first integral of angular momentum $\det(q,\dot q)$.

Kepler's system enjoys also more recondite infinitesimal invariances, which lead to the Laplace-Runge-Lenz vector constant of motion. Next we will describe two different invariances: one that needs the Lagrange equation, and one that does not.

Let us take again a smooth trajectory~$q(t)$ in~$\R^2\setminus \{(0 ,0)\}$, a vector $u\in\R^2$, and define the family
\begin{equation}\label{familyrungelenz}
  \begin{split}
  q_\varepsilon(t):={}&q(t)+
  \bigl(q(t)\cdot u\bigr)
  q(t+\varepsilon)-
  \bigl(q(t+\varepsilon)\cdot u\bigr) q(t)=\\
  ={}&q(t)+\det\bigl(q(t),q(t+\varepsilon)\bigr)u^\perp
  \end{split}
\end{equation}
with $u^\perp=(\begin{smallmatrix}0&-1\\1&0\end{smallmatrix})u$. It is clear that $q_0(t)=q(t)$. This family $q_\varepsilon$ is different and simpler (even in its first order $\varepsilon$-expansion) than the one found in the literature (for example, see L\'evy-Leblond~\cite{Leblond}, formula~(36)), whose formula in our setting would read as
\begin{equation*}
 q(t)+\frac{\varepsilon}{2}
 \Bigl(2 \bigl(q(t)\cdot u\bigr) \dot q(t)-
 \bigl(\dot q(t)\cdot u\bigr) q(t)
 -\bigl(\dot q(t)\cdot q(t)\bigr) u\Bigr).
\end{equation*}
The reader can check that for our $q_\varepsilon$ the following relations hold:
\begin{equation*}
  \frac{\partial}{\partial\varepsilon}
  L\bigl(q_\varepsilon(t),\dot q_\varepsilon(t)\bigr)
  \Big|_{\varepsilon=0}=
  \det\bigl(q(t),\ddot q(t)\bigr)\det\bigl(u,\dot q(t)\bigr)+
  \frac{d}{dt}\biggl(k
  \frac{q(t)\cdot u}{\lVert q(t)\rVert}
  \biggr)\,.
\end{equation*}
If we define the BH-function
\begin{equation}\label{KeplerGauge}
 G(\varepsilon,t):=-\varepsilon k
 \frac{q(t)\cdot u}{\lVert q(t)\rVert},
\end{equation}
and if $q(t)$ is any motion for which $\ddot q(t)$ is parallel to $q(t)$, then we have infinitesimal invariance of the form~\eqref{gaugeAndSpace}:
\begin{equation*}
  \frac{\partial}{\partial\varepsilon}\biggl(
  L\bigl(q_\varepsilon(t),\dot q_\varepsilon(t)\bigr)+
  \partial_tG(\varepsilon,t)\biggr)
  \biggr|_{\varepsilon=0}\equiv 0\,,
\end{equation*}
which is covered by Theorem~\ref{NoetherWithSpaceChangeAndGauge}, condition~\ref{timeChangeAndGaugeCondition}, with the trivial $\tau(\varepsilon,t)\equiv t$. Also, $\ddot q(t)$ and $q(t)$ are parallel whenever $q(t)$ is a solution to Kepler's equation~\eqref{KeplerEquation}. Therefore Noether's Theorem~\ref{NoetherWithSpaceChangeAndGauge} (or Th.~\ref{NoetherTheoremWithTotalDerivative}) yields the following constant of motion:
\begin{equation}\label{rungeLenzComponent}
  (q\cdot u)\lVert\dot q\rVert^2-
  (\dot q\cdot u) (\dot q\cdot q) -
  k\frac{q\cdot u}{\lVert q\rVert}.
\end{equation}
Since the vector $u\in\R^2$ is arbitrary, we have the vector-valued first integral
\begin{equation*}
  q\lVert\dot q\rVert^2- (\dot q\cdot q)\dot q -
  k \frac{q}{\lVert q\rVert}\,,
\end{equation*}
which is called the Laplace-Runge-Lenz vector.

According to Theorem~\ref{EquivalenceOfThreeVersionsTheorem}, if we define the additional time change~$\mathcal{T}$ as
\begin{equation*}
  \mathcal{T}(\varepsilon,t):=\varepsilon
  \frac{\partial_\varepsilon G(0,t)}{L\bigl(q(t),
  \dot q(t)\bigr)}=
  -\varepsilon k \frac{q(t)\cdot u}{\lVert q(t)\rVert
  L\bigl(q(t),\dot q(t)\bigr)}\,.
\end{equation*}
there is infinitesimal invariance also under the nontrivial time change $(\varepsilon,t)\mapsto t+\mathcal{T}(\varepsilon,t)$, without BH-function.

A second way of deducing the Laplace-Runge-Lenz vector from Noether's theorem is described by Sarlet and Cantrijn~\cite[Sec.~6]{SarletCantrijn}, who use a general idea introduced by Djukic~\cite{Djukic} (see also Kobussen~\cite{Kobussen}). It has the advantage that the total derivative condition does not need the equations of motion. Translated into our notation, Sarlet and Cantrijn use the following triple space change~$q_\varepsilon$, time change $\theta_\varepsilon$ (here given in the form of Subsection~\ref{alternativeApproachToTimeSubsection}) and BH-function~$G$:
\begin{gather}
  \Xi(q,\dot q):=\begin{pmatrix}u_1&-u_2\\u_2&u_1\end{pmatrix}
  \binom{q\cdot \dot q}{
  \det(\dot q, q)},\quad
  q_\varepsilon(t):=q(t)+\varepsilon
  \Xi\bigl(q(t),\dot q(t)\bigr),\notag\\
  \theta_\varepsilon(t):=t+\varepsilon q(t)\cdot u,\quad
  \psi(q,\dot q):=\frac{1}{2}\dot q\cdot\Xi(q,\dot q),\\
  G(\varepsilon,t):=-\varepsilon\psi\bigl(q(t),\dot q(t)\bigr),
  \notag
\end{gather}
where again $u=(u_1,u_2)\in\R^2$ is a constant vector. The invariance is the one in Definition~\ref{alternativeGaugeInvarianceUnderSpaceAndTimeChangeDef}, formula~\eqref{alternativeGaugeInvarianceSpaceAndTimeLagrangian}, or, in terms of total derivative,
\begin{equation*}
  \frac{\partial}{\partial\varepsilon}\biggl(
  L\Bigl(q_\varepsilon(t),
  \frac{\dot q_\varepsilon(t)}{\theta'_\varepsilon(t)}
  \Bigr)\theta'_\varepsilon(t)
  \biggr)\bigg|_{\varepsilon=0}=
  \frac{d}{dt}\psi\bigl(q(t),\dot q(t)\bigr).
\end{equation*}
As already noted, unlike the previous derivation starting from the space chan\-ge~\eqref{familyrungelenz}, this total derivative identity holds for all smooth $q(t)$, without any need to use the Lagrange equations. From Theorem~\ref{alternativeFullNoetherWithSpaceAndTimeAndGauge}
we deduce the constant of motion
\begin{equation*}
  \begin{split}
  N={}& \partial_{\dot q}L\cdot
  \Bigl(\partial_\varepsilon q_\varepsilon|_{\varepsilon=0}-
  (\partial_\varepsilon\theta_\varepsilon|_{\varepsilon=0})
  \dot q\Bigr)+L\cdot
  \partial_\varepsilon\theta_\varepsilon|_{\varepsilon=0}+
  \partial_\varepsilon
  G|_{\varepsilon=0}=\cr
  ={}&-(q\cdot u)\lVert\dot q\rVert^2+
  (\dot q\cdot u) (\dot q\cdot q) +
  k\frac{q\cdot u}{\lVert q\rVert}\,,
  \end{split}
\end{equation*}
which coincides with the opposite of formula~\eqref{rungeLenzComponent}.

Here too we can easily ``trivialize'' either the time change or the BH-function, but we have to use Theorem~\ref{EquivalenceOfThreeVersionsUsingInverseOfThetaTheorem}
instead of Theorem~\ref{EquivalenceOfThreeVersionsTheorem}
because of the different style of the time change. For example, if we define
\begin{equation*}
  \mathcal{T}_\varepsilon(t):=\varepsilon\cdot
  \frac{\partial_\varepsilon G(\varepsilon,t)|_{\varepsilon=0}
  }%
  {L\bigl(q(t),
  \dot q(t)\bigr)}=
  -\varepsilon\cdot\frac{\psi\bigl(q(t),\dot q(t)\bigr)}
  {L\bigl(q(t),\dot q(t)\bigr)}\,,
\end{equation*}
we can eliminate the BH-function with this choice of space change $\mathcal{Q}_\varepsilon$ and time change $\Theta$:
\begin{equation*}
  \Theta_\varepsilon:=
  \theta_\varepsilon+\mathcal{T}_\varepsilon=
  \theta_\varepsilon-\varepsilon\frac{\psi}{L},\quad
  \mathcal{Q}_\varepsilon=
  q_\varepsilon+\varepsilon\cdot
  \bigl(\partial_\epsilon\mathcal{T}_\epsilon
  |_{\epsilon=0}\bigr)\dot q=
  q_\varepsilon-\varepsilon\cdot\frac{\psi}{L}\dot q.
\end{equation*}
Here too we have infinitesimal invariance without using the equations of motion.

Let us show a simple example where Noether's theorem is applied to a \emph{particular} solution, leading to a nontrivial function that is constant along that single solution, but not along most others. Consider the following family of uniform circular motions in the plane with the same period but a phase shift:
\begin{equation*}
  q_\varepsilon(t):=e^{\omega\varepsilon}R \bigl(\cos
  (\omega(\varepsilon+t)),\sin (\omega(\varepsilon+t))\bigr),
\end{equation*}
where $R>0$, $\omega=\sqrt{k/R^3}$. The function $t\mapsto q_\varepsilon(t)$ is a Kepler motion only when $\varepsilon=0$, as we check at once. Still, there is infinitesimal invariance with trivial $\tau(\varepsilon,t)\equiv0$, $G\equiv0$:
\begin{equation*}
  \frac{\partial}{\partial\varepsilon}
  L\bigl(q_\varepsilon(t),\dot q_\varepsilon(t)\bigr)
  \Big|_{\varepsilon=0}%=\\
  =\frac{\partial}{\partial\varepsilon}
  \Bigl(\frac{R^2\omega^2}{2}
  e^{2\omega \varepsilon}+\frac{k}{Re^{\omega\varepsilon}}\Bigr)
  \Big|_{\varepsilon=0}=0
\end{equation*}
If we apply Noether's theorem we obtain that the (square of) the speed $t\mapsto \lVert\dot q_0(t)\rVert^2$ is constant along the circular motion. The speed is clearly a nontrivial function that is not constant along any Kepler motions except circular ones.

%%%%%%%%%%%%%%%%%%%%%%%%
\section{New superintegrable systems}
\label{sectionisochronous}

A recent paper~\cite{Zampieri} introduced the following Lagrangian systems in two dimensions:
\begin{equation*}
  L(x,y,\dot x,\dot y)=\dot x\dot y-g(x)y,
\end{equation*}
and exhibited some classes of the function~$g$ for which there is either weak Lyapunov instability or an isochronous center. What matters here is that those isochronous cases exhibit the rare property of being superintegrable, because they have three independent first integrals, out of four degrees of freedom. Here we are going to use Noether's theorem framework to compute new $g$ classes that lead to super-integrability.

Two first integrals are obvious, and they do not need any assumption on~$g$: $\frac12\dot x^2+V(x)$ and $\dot y\dot x+g(x) y$, where $V$ is any primitive of~$g$. Notice that the Lagrange equations are
\begin{equation}\label{lagrangeEquationsForIsochrony}
  \ddot x=-g(x)\,,\qquad
  \ddot y=-g'(x)y\,.
\end{equation}

Let us search for a third constant of motion starting from the following space-change family:
\begin{equation*}
  q_\varepsilon(t)=\Bigl(x(t)+\varepsilon f(x(t),\dot x(t)),
  y(t)\Bigr).
\end{equation*}
where $f$ is a function to be determined. We can compute, using also the Lagrange equations~\eqref{lagrangeEquationsForIsochrony}:
\begin{equation*}
  \frac{\partial}{\partial\varepsilon}
  L(q_\varepsilon(t), \dot q_\varepsilon(t))
  \big|_{\varepsilon=0}%=\\
  =\dot y\dot f -g'(x)yf=\frac{d}{dt}\bigl(y\dot f\bigr)-
  \bigl(\ddot f+g'(x)f\bigr)y\,.
\end{equation*}
If we impose that
\begin{equation}\label{conditionForTotalDerivative}
  \ddot f+g'(x)f\equiv0
\end{equation}
then we achieve the total derivative condition, there is infinitesimal BH-invari\-ance with the BH-function~$G(\varepsilon,t)=-\varepsilon y(t)\dot f$, and Noether's theorem gives the following first integral
\begin{equation}\label{additionalFirstIntegral}
  \partial_{\dot q}L\bigl(q(t),\dot q(t)\bigr)
  \cdot \partial_\varepsilon
  q_\varepsilon(t)|_{\varepsilon=0}+
  \partial_\varepsilon
  G(\varepsilon,t)
  |_{\varepsilon=0}=\dot y f-y\dot f\,.
\end{equation}
It is the first case so far where the BH-function~$G$ is a function of~$\dot q$, and not just of $t$ and~$q$. One more example will be in the next Section.

The condition~\eqref{conditionForTotalDerivative} is a partial differential equation in the unknown function~$f(x,\dot x)$. We will search for smooth solutions that are even functions of~$\dot x$:
\begin{equation}\label{ansatz}
  f(x,\dot x)=\alpha(x)+\dot{x}^2 \beta(x)
  +\dot{x}^4 h(x,\dot{x}^2)
\end{equation}
for new unknown functions $\alpha(x),\beta(x),h(x,\dot x)$. Let us replace the ansatz~\eqref{ansatz} into the equation~\eqref{conditionForTotalDerivative} and set $\dot x=0$. We get
\begin{equation}\label{conditionOnAlpha/g}
  0=\alpha (x) g'(x)-g(x)\alpha'(x)+2 g(x)^2 \beta (x)=
  g(x)^2\Bigl(2\beta(x)-\frac{d}{dx}\bigl(\alpha(x)/g(x)\bigr)
  \Bigr)
\end{equation}
This suggests to introduce the new function $\mu(x)=\alpha(x)/g(x)$, so that equation~\eqref{conditionOnAlpha/g} becomes simply $\beta=\mu'/2$. We can change the ansatz~\eqref{ansatz} into the more special form
\begin{equation*}
  f(x,\dot x)=g(x)\mu(x)+\frac{1}{2}\dot{x}^2 \mu'(x)
  +\dot{x}^4 h(x,\dot{x}^2)\,.
\end{equation*}
Let us try with the choice $h\equiv a$ constant. Equation~\eqref{conditionForTotalDerivative} becomes
\begin{multline*}
  \frac{1}{2} \dot{x}^4 \bigl(\mu'''(x)-6 a
   g'(x)\bigr)+\\
   +\frac{1}{2} \dot{x}^2
   \Bigl(24 a g(x)^2+2 \mu (x) g''(x)
   +3 g'(x) \mu '(x)-
   3 g(x)\mu''(x)\Bigr)=0\,,
\end{multline*}
which splits into the equations
\begin{equation*}
  \begin{cases}
  \mu'''(x)-6 a g'(x)=0\\
  24 a g(x)^2+2 \mu (x) g''(x)+3 g'(x) \mu '(x)-
   3 g(x)\mu''(x)=0\,,
  \end{cases}
\end{equation*}
to which we can apply the standard local existence results for~ODEs. Any solution of this equations will give a function~$g$ for which the system is superintegrable. The additional first integral~\eqref{additionalFirstIntegral} becomes
\begin{equation*}
  a \bigl(4 \dot{x}^3 y g(x)+\dot{x}^4 \dot{y}\bigr)
  -\frac{1}{2} \dot{x}^3 y\mu ''(x)
  %+{}\\
  +\frac{1}{2} \dot{x}^2 \dot{y} \mu '(x)
  -\dot{x} y \mu (x) g'(x)
  +\dot{y} g(x) \mu (x)\,.
\end{equation*}
It can be checked that the special case $a=0$, with the initial data $g(0)=0$, $g'(0)>0$, gives exactly the class of  systems that were found in Section~5, entitled ``Explicit superintegrable systems'', of the original paper~\cite{Zampieri}, for which the additional first integral is a polynomial of degree~3 with respect to~$\dot{x}$. The case $a\ne0$ brings superintegrable systems with a first integral containing a term in~$\dot{x}^4$.

With some patience it is possible to pursue these ideas further, for example by studying the case when $h(x,\dot x^2)$ is of the form $\gamma(x)+a\dot x^2$.

%%%%%%%%%%%%%%%%%%%%%%%%
\section{Particle in a plane-wavelike external field}
\label{PlaneWaveLikeExample}

The following time-dependent Lagrangian is taken from a paper~\cite{Bobillo} by Bobil\-lo-Ares:
\begin{equation*}
  L(t,q,\dot q)=
  \frac{1}{2}\lVert\dot q\rVert^2-U\bigl(q-tu\bigr),\qquad
  q,\dot q\in\R^n,
\end{equation*}
where $u$ is a fixed vector in $\R^n$ and $U$ a smooth potential. The associated Lagrange equation is
\begin{equation}\label{plane-wave-likeLagrangeEquations}
  \ddot q+\nabla U\bigl(q-tu\bigr)=0
\end{equation}
In terms of the energy 
\begin{equation*}
  E(t,q,\dot q)=\partial_{\dot q}L(t,q,\dot q)\cdot\dot q
  -L(t,q,\dot q)%=\\
  =\frac{1}{2}\lVert\dot q\rVert^2+U\bigl(q-tu\bigr)\,,
\end{equation*}
it is easy to check that $\dot q\cdot u-E$ is a first integral. Let us see how we can deduce it from Noether's theorem in our framework.  Starting from a smooth~$q(t)$ and following Bobillo-Ares, we introduce the following space and time changes:
\begin{equation}\label{bobilloInfinitesimalInvariance}
  q_\varepsilon(t)=q(t)+\varepsilon u,\qquad
  \tau(\varepsilon,t):=t+\varepsilon\,.
\end{equation}
Let us try infinitesimal invariance:
\begin{align*}
  \frac{\partial}{\partial\varepsilon}\Bigl(&
  L\bigl(\xi,q_\varepsilon(\xi),\dot q_\varepsilon(\xi)\bigr)
  \Big|_{\xi=\tau(\varepsilon,t)}
  \partial_t\tau(\varepsilon,t)
  \Bigr)=\notag\\
  ={}& \frac{\partial}{\partial\varepsilon}\Bigl(
  \frac{1}{2}\lVert\dot
  q(t+\varepsilon)\rVert^2-
  U\bigl(q(t+\varepsilon)+\varepsilon
  u-(t+\varepsilon)u\bigr)\Bigr)=\notag\\
  ={}&\dot q(t+\varepsilon)\cdot\ddot q(t+\varepsilon)-
  \nabla U\bigl(q(t+\varepsilon)-tu\bigr)\cdot
  \dot q(t+\varepsilon)=\notag\\
  ={}&\dot q(t+\varepsilon)\cdot\biggl(
  2\ddot q(t+\varepsilon)-
  \Bigl(\ddot q(t+\varepsilon)+
  \nabla U\bigl(q(t+\varepsilon)-tu\bigr)\Bigr)\biggr)=\notag\\
  ={}&\frac{\partial}{\partial t}
  \lVert\dot q(t+\varepsilon)\rVert^2-
  \dot q(t+\varepsilon)\cdot
  \Bigl(\ddot q(t+\varepsilon)+
  \nabla U\bigl(q(t+\varepsilon)-tu\bigr)\Bigr)
  \,.\notag
\end{align*}
When $\varepsilon=0$ this expression becomes
\begin{equation*}
  \frac{d}{dt}
  \lVert\dot q(t)\rVert^2
  -\dot q(t)\cdot \Bigl(\ddot q(t)+
  \nabla U\bigl(q(t)-tu\bigr)\Bigr)\,.
\end{equation*}
which further reduces to
\begin{equation*}
  \frac{\partial}{\partial\varepsilon}\Bigl(
  L\bigl(\xi,q_\varepsilon(\xi),\dot q_\varepsilon(\xi)\bigr)
  \Big|_{\xi=\tau(\varepsilon,t)}
  \partial_t\tau(\varepsilon,t)
  \Bigr)\Big|_{\varepsilon=0}%=\\
  =\frac{d}{dt}
  \lVert\dot q(t)\rVert^2\,.
\end{equation*}
if $q(t)$ solves Lagrange equations~\eqref{plane-wave-likeLagrangeEquations}. This means that we have infinitesimal invariance as in Theorem~\ref{fullNoetherWithSpaceAndTimeAndGauge} with the following choice of BH-function:
\begin{equation*}
  G(\varepsilon,t):=-\varepsilon \lVert\dot q(t)\rVert^2\,,
\end{equation*}
and the first integral \eqref{constantAlongMotionForSpaceTimeChangeAndGauge}
given by Noether's theorem is
\begin{equation*}
  \dot q(t)\cdot u+L(t,q(t),\dot q(t))-\lVert\dot q(t)\rVert^2\,.
\end{equation*}
as expected.
This is one more example where the BH-function~$G$ depends on~$\dot q$, so that is not of the more familiar form~$\varepsilon\cdot\psi(t,q(t))$.

According to Theorem~\ref{EquivalenceOfThreeVersionsTheorem}, if we define the additional time change and BH-function
\begin{align*}
  \mathcal{G}(\varepsilon,t):={}&\varepsilon\cdot L\bigl(t,
  q(t),\dot q(t)\bigr)
  \bigl(\partial_\epsilon
  \tau(\epsilon,t)|_{\epsilon=0}\bigr)=
  \varepsilon\cdot L\bigl(t,
  q(t),\dot q(t)\bigr)\,,
  \label{auxiliaryGaugeForBobillo}\\
  \mathcal{T}(\varepsilon,t):={}&\varepsilon
  \frac{\partial_\varepsilon G(0,t)}{L\bigl(t,q(t),
  \dot q(t)\bigr)}=
  -\varepsilon\frac{\lVert\dot q(t)\rVert^2}
  {L\bigl(t,q(t),
  \dot q(t)\bigr)}\,,
\end{align*}
we can attain infinitesimal invariance also with either of the alternative choices
\begin{gather*}
  \tau_1(\varepsilon,t):=\tau(\varepsilon,t)+
  \mathcal{T}(\varepsilon,t)\,,
  \qquad G_1(\varepsilon,t)\equiv0\,,\\
  \tau_2(\varepsilon,t):=t\,,
  \qquad G_2(\varepsilon,t):= G(\varepsilon,t)+
  \mathcal{G}(\varepsilon,t)\,,
\end{gather*}
for the same space change~$q_\varepsilon$.

A possible alternative choice for the BH-function term in~\eqref{bobilloInfinitesimalInvariance} is the following
\begin{equation*}
  G_3(\varepsilon,t):=
  \begin{cases}
  -\bigl\|
  q(t+\varepsilon)-q(t)\bigr\|^2/\varepsilon
  &\text{if }\varepsilon\ne0\\
  0&\text{if }\varepsilon=0\,,
  \end{cases}
\end{equation*}
which is not linear in~$\varepsilon$, and is not a point function of~$q(t),\dot q(t)$.

%%%%%%%%%%%%%%%%%%%%%%%%%%%%

\end{document}